\documentclass{article}

\usepackage{arxiv}

\usepackage[utf8]{inputenc} 
\usepackage[T1]{fontenc}    
\usepackage{hyperref}       
\usepackage{url}            
\usepackage{booktabs}       
\usepackage{amsfonts}       
\usepackage{graphicx}
\usepackage{doi}
\usepackage{amsmath,amssymb}
\usepackage{amsthm}
\usepackage{epstopdf}
\usepackage{placeins}
\usepackage{subfig}
\usepackage{mathtools}
\usepackage{longtable}
\usepackage{bm}
\usepackage{tikz-cd}
\usepackage{multicol}
\usepackage{tikz}
\usepackage{tikz-3dplot}
\usepackage{xifthen}
\usepackage{placeins}
\usepackage{makecell}

\usepackage[numbers,sort&compress]{natbib}

\usepackage{ifthen}
\usepackage[nomessages]{fp}
\usepackage{diagbox}
\usepackage{stmaryrd}
\usepackage{algpseudocode}  
\usepackage{setspace} 
\usepackage{booktabs} 
\usepackage{esint}
\usepackage{xcolor}
\usepackage{amsmath,amssymb}
\usepackage{bm}
\usepackage{mathtools}
\usepackage{tikz}
\usepackage{tikz-cd}
\usepackage{tikz-3dplot}
\newcommand{\Rz}{\mathbb{R}}
\newcommand{\Rzsp}{\mathbb{R}_{>0}}

\newcommand{\Kn}{K\!n}
\newcommand{\Ma}{M\!a}
\newcommand{\Rey}{R\!e}

\newcommand{\matD}{\frac{\mathrm{D}}{\mathrm{D} t}}
\newcommand{\matDil}{\mathrm{D}/(\mathrm{D} t)}
\usepackage{amsopn}

\allowdisplaybreaks

\newtheorem{theorem}{Theorem}

\newtheorem{proposition}{Proposition}
\newtheorem{lemma}{Lemma}

\newtheorem{conjecture}{Conjecture}

\theoremstyle{definition}
\newtheorem{definition}{Definition}

\newtheorem{remark}{Remark}
\newtheorem{assumption}{Assumption}

\title{Homogenized lattice Boltzmann methods for fluid flow through porous media -- part I: kinetic model derivation}

\author{
    \href{https://orcid.org/0000-0001-8555-4245}{\includegraphics[scale=0.06]{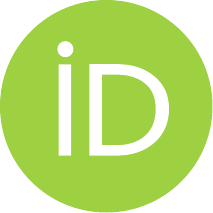}\hspace{1mm}Stephan Simonis}\thanks{Corresponding author}\\
	Institute for Applied and Numerical Mathematics\\
    Karlsruhe Institute of Technology\\
    76131 Karlsruhe, Germany \\
	\texttt{stephan.simonis@kit.edu} \\
	\And
    Nicolas Hafen \\
	Institute of Mechanical Process Engineering and Mechanics \\
    Karlsruhe Institute of Technology \\
    76131 Karlsruhe, Germany \\
	\And
    Julius Je{\ss}berger \\
    Lattice Boltzmann Research Group \\
    Karlsruhe Institute of Technology \\
    76131 Karlsruhe, Germany \\
    \And 
    \href{https://orcid.org/0000-0002-3442-6857}{\includegraphics[scale=0.06]{orcid.pdf}\hspace{1mm}Davide Dapelo} \\
    School of Engineering \\
    University of Liverpool \\
    L69 3BX Liverpool, United Kingdom
    \And
    Gudrun Th\"{a}ter \\
    Institute for Applied and Numerical Mathematics \\
    Karlsruhe Institute of Technology \\
    76131 Karlsruhe, Germany \\
    \And 
    \href{https://orcid.org/0000-0003-1026-6462}{\includegraphics[scale=0.06]{orcid.pdf}\hspace{1mm}Mathias J. Krause}\\
    Lattice Boltzmann Research Group \\
    Karlsruhe Institute of Technology \\
    76131 Karlsruhe, Germany \\
}

\renewcommand{\headeright}{S.\ Simonis \emph{et al.}}
\renewcommand{\shorttitle}{HLBM for fluid flow through porous media -- part I}

\hypersetup{
pdftitle={HLBMforPorousMediaFlowPartI},
pdfsubject={},
pdfauthor={Stephan Simonis}
}

\begin{document}
\maketitle

\begin{abstract}
In this series of studies, we establish homogenized lattice Boltzmann methods (HLBM) for simulating fluid flow through porous media.
Our contributions in part I are twofold.
First, we assemble the targeted partial differential equation system by formally unifying the governing equations for nonstationary fluid flow in porous media.
A matrix of regularly arranged, equally sized obstacles is placed into the domain to model fluid flow through porous structures governed by the incompressible nonstationary Navier--Stokes equations (NSE).
Depending on the ratio of geometric parameters in the matrix arrangement, several homogenized equations are obtained.
We review existing methods for homogenizing the nonstationary NSE for specific porosities and discuss the applicability of the resulting model equations.
Consequently, the homogenized NSE are expressed as targeted partial differential equations that jointly incorporate the derived aspects.
Second, we propose a kinetic model, the homogenized Bhatnagar--Gross--Krook Boltzmann equation, which approximates the homogenized nonstationary NSE.
We formally prove that the zeroth and first order moments of the kinetic model provide solutions to the mass and momentum balance variables of the macrocopic model up to specific orders in the scaling parameter. 
Based on the present contributions, in the sequel (part II), the homogenized NSE are consistently approximated by deriving a limit consistent HLBM discretization of the homogenized Bhatnagar--Gross--Krook Boltzmann equation.
\end{abstract}



\section*{List of Symbols}

\begin{longtable}{p{0.13\linewidth} p{0.8\linewidth}}
	\midrule
        Expression & Meaning \\
    \midrule 
       DL & Darcy's law \\
       BL & Brinkman law \\
       NSE & Navier--Stokes equations \\
       LBM & lattice Boltzmann method \\
       LBE & lattice Boltzmann equation \\
       BGK & Bhatnagar--Gross--Krook \\
       BGKBE & Bhatnagar--Gross--Krook Boltzmann equation \\
       DVBE & discrete velocity BGK Boltzmann equation \\
       HNSE & homogenized Navier--Stokes equations \\
       HLBM & homogenized lattice Boltzmann method \\
       HLBE & homogenized lattice Boltzmann equation \\
       HBGKBE & homogenized BGK Boltzmann equation \\
       HDVBE & homogenized discrete velocity BGK Boltzmann equation \\
       TEQ & target equation \\
       \(\Omega\) & domain of the porous media including solid and fluid regions \\
       \(\Omega_{\epsilon}\) & fluid void filling the porous media structure \\
       \(\partial \Omega\) & Boundary of \(\Omega\) \\
       \(d\) & dimension, \(\Omega \in \mathbb{R}^{d}\) \\
       \(Y_{i}^{\epsilon}\) & \(i\)th cell in porous structure \\
       \(Y_{S,i}^{\epsilon}\), \(Y_{F,i}^{\epsilon}\) & \(i\)th spherical obstacle where \(1\leq i\leq N(\epsilon)\); and \(i\)-th fluid void cell \\
       \(N(\epsilon)\) & number of solid obstacles in the porous structure \\
       \(H^1(X)\) & Sobolev space \(H^{k}(X) = W^{k,2}(X)\), where \(k=1\) \\
       \(H^1_0(X)\) & functions \(f \in H^{1}(X)\) with vanishing trace \(f\vert_{\partial X} = 0\) \\
       \(H^{1}_{\mathrm{div}}(X)\) & divergence-free functions \(f \in H^{1}(X)\) \\
       \(H^{1}_{\#}(X)\) & \(X\)-periodic functions in \(H^{1}(X)\) \\
       \(\epsilon\) & side length of geometric porous media cell containing one obstacle \\
       \(a_{\epsilon}\) & size or diameter of solid obstacle \\
       \(a_{\epsilon}^{\mathrm{crit}}\) & critical  obstacle size \\
       \(\sigma_{\epsilon}\) & ratio function of cell side length and obstacle size \\
       \(\bm{u}_{\epsilon}\), \(p_{\epsilon}\) & fluid velocity and pressure on cell scale (nonhomogenized) \\
       \(\tilde{\bm{u}}_{\epsilon}\), \(\tilde{p}_{\epsilon}\) & extension of the solution \(\bm{u}_{\epsilon}\), \(p_{\epsilon}\) \\
       \(\bm{F}\) & given force field \\
       \(\nu\) & kinematic viscosity \\
       \(C_{i}^{\epsilon}\) & control volume containing \(Y_{S,i}^{\epsilon}\) \\
       \(\mathbf{M}\) & porosity matrix \\
       \(\bm{e}_{k}\) & \(k\)th unit basis vector of \(\mathbb{R}^{d}\) \\
       \(\iota_{i}^{\epsilon}\) & linear homeomorphism, mapping each cell to the unit cell \\
       \(\bm{w}_{k}\), \(q_{k}\) & fluid velocity and pressure in \(k\)th stationary local model problem \\
       \(\mathbf{A}\) & permeability tensor \\
       \(\bm{v}_{k}\), \(p_{k}\) & fluid velocity and pressure in \(k\)th stationary unit cell problem \\
       \(Y\) & unit cell \\
       \(Y_{S}\), \(Y_{F}\) & solid part and fluid part of unit cell \\
       \(Y^{m}_{S}\) & model obstacle in the model problem \\
       \(\delta\) & scaling prefactor for the case \(a_{\epsilon} = \delta \epsilon\) \\
       \(\bm{n}\) & outward pointing normal vector \\
       \(\bm{w}^{j}\), \(\pi^{j}\) & fluid velocity and pressure in \(j\)th nonstationary unit cell problem \\
       \(\tilde{\mathbf{A}}(t)\) & time-dependent permeability tensor \\
       \(C\) & scaling constant for the cases \(a_{\epsilon} = C\epsilon^{n}\), where \(n\in \mathbb{N}\) \\
       \(\sigma\) & constant limit value of ratio \(\sigma_{\epsilon}\) in case of \(a_{\epsilon} = \mathcal{O}(\epsilon^{3})\) \\
       \(\varphi\) & porosity \\
       \(A\) & Eigenvalue of isotropic permeability tensor \\
       \(f\) & Particle density function \\
       \(\bm{c}_i\) & \(i\)th discrete velocity \\
       \(\tilde{\bm{c}}_i\) & \(i\)th prefactored discrete velocity \\
    \midrule
\end{longtable}

\section{Introduction}
The governing equations for fluid flow in porous media typically consist of modified versions of the Navier--Stokes equations (NSE). 
Several mathematical models exist, based on the type of application. 
Depending on the context of porous media flows, most models can be categorized as either mathematically-motivated, or application-related. 

For the mathematical modeling of fluid flow through porous media, the incompressible NSE can be modified to include the effects of the solid matrix on the fluid flow in the void.
Various different mathematical models exist (see \cite{hornung1997homogenization,nield2017convection}, and references therein). 
Here, we recall the rigorous construction of porous media flow models formulated in Allaire's seminal works, see e.g.\ \cite{allaire1989homogeneisation,allaire1991homogenization,allaire1991homogenizationII,allaire1991homogenizationNSE,allaire1991continuity,allaire1992homogenization,allaire1992homogenizationAndTwo,allaire2010homogenizationLec2}. 
Therein, the geometric definition of porous media as sets of equidistant obstacles in the flow domains is considered to construct model equations via homogenization. 
As a result, several homogenization limits are derived, whereby the homogenized equations depend on the geometric configuration. 
We distinguish between three classical cases of homogenization limits: 
\begin{itemize}
\item incompressible NSE,
\item Brinkman law (BL), 
\item Darcy's law (DL). 
\end{itemize}
The respective limits in this categorization were rigorously proved for the stationary \cite{allaire1992homogenizationAndTwo} and nonstationary Stokes regime \cite{allaire1992homogenization}, as well as for the stationary NSE \cite{allaire1991homogenizationNSE} as starting points. 
Although suggested by Allaire, to the knowledge of the authors, the validity of the stationary categorization of homogenization limits is not completely proven for the nonstationary NSE.  
Nevertheless, the works of Mikeli\'{c}~\cite{mikelic1991homogenization,mikelic1994mathematical} and Feireisl \textit{et al.}~\cite{feireisl2016homogenization} cover the homogenization limit toward the BL and the DL in the non-stationary case in a different framework. 
Other contributions also used this structural categorization, see e.g.\ \cite{griebel2009homogenisation,klitz2006homogenisation,laptev2003numerical}. 
Although these models are likely to be interconnected, rigorous proofs of the underlying relations are rare and limited to linear and stationary settings. 
For instance, Allaire~\cite{allaire1991continuity} proved the compliance of a formally derived DL and the DL derived via homogenization (low volume fraction limit). 
Feppon~\cite{feppon2021high} and Feppon \textit{et al.}~\cite{feppon2022allThree} proved high-order homogenization limits for the Stokes equations in a unified procedure. 
To the knowledge of the authors, the latter is the first and only derivation covering all three classical cases together with the low volume fraction limit at once. 
However, it should be noted that these unified studies have not been conducted for homogenizing the nonstationary NSE, yet.

Besides the mathematically rigorous model derivation, application-based model construction has been found to be suitable for fluid flow in porous media \cite{nithiarasu1997natural,spaid1997lattice,guo2002lattice}.
Typically, empirical observations and matching terms are used to introduce model systems akin to Brinkman-~\cite{brinkman1949permeability}, Forchheimer-~\cite{forchheimer1901wasserbewegung}, Darcy-~\cite{darcy1856fontaines}, or mixed-type equations \cite{nield2017convection}. 
Depending on the characteristic scales of porosity in the application in question, the heuristically derived models can correctly recover the flow physics or severely disagree with experiments \cite{nield2017convection}. 
However, due to the large variation of involved spatial scales, the model equations are often solved numerically with highly parallelizable methods.  
For example, Spaid and Phelan~\cite{spaid1997lattice} proposed a lattice Boltzmann method (LBM) for approximating Stokes and Stokes--Brinkman equations as target models. 
The latter only apply to large-size obstacles in the porous matrix and solely recover stationary flows. 
The LBM meanwhile is an established numerical technique for the approximate solution of various transport problems \cite{lallemand2021lattice}. 
Providing distinct advantages in terms of parallelizability, the LBM is well-suited for computational fluid dynamics and multiphysics simulations where good scalability on high-performance computing (HPC) facilities is crucial \cite{krause2020openlb,simonis2020relaxation,haussmann2019direct,simonis2022binary,simonis2022forschungsnahe,bukreev2023consistent,simonis2022constructing,simonis2022forschungsnahe,mink2021comprehensive,dapelo2021lattice,simonis2023pde,haussmann2021fluid,simonis2021linear,simonis2022temporal,siodlaczek2021numerical}. 
Even standard LBM formulations offer an easy to implement and mostly second order accurate, intrinsically matrix-free algorithm in space-time. 
Those are well-suited for approximating nonstationary and nonlinear problems and, if optimized properly, also capable of saturating modern-day HPC machinery \cite{kummerlander2022advances,kummerlander2022implicit}. 
To the knowledge of the authors, the LBM has not been used yet to approximate the nonstationary homogenized NSE which governs time-dependent and nonlinear (possibly turbulent) fluid flow through abstracted porous media. 

Consequently, the overall aim of this series of works is to construct homogenized LBMs (HLBMs) that approximate the governing equations for homogenized nonstationary nonlinear fluid flow through porous media. 

Our contributions in part I are twofold. 
First, we assemble the targeted partial differential equation (PDE) system by formally unifying the governing equations for nonstationary fluid flow in homogenized porous media. 
To this end, a matrix of regularly arranged obstacles of equal size is placed into the domain to model fluid flow through structures of different porosities that is governed by the incompressible nonstationary NSE. 
Depending on the ratio of geometric parameters in the matrix arrangement, several cases of homogenized PDEs are obtained. 
We review the existing methods to homogenize the stationary NSE. 
From that we assemble a conjecture for the cases of PDE models resulting from homogenization of the nonstationary NSE for specific porosities. 
Moreover, we interpret connections between the resulting model equations from the perspective of applicability. 
Consequently, the homogenized nonstationary NSE are formulated as unified targeted PDE system which jointly incorporates the derived aspects. 
Second, we propose a kinetic model, named homogenized Bhatnagar--Gross--Krook (BGK) Boltzmann equation, which approximates the homogenized nonstationary NSE in a diffusive scaling limit. 
We formally prove that the zeroth and first order moments of the kinetic model provide solutions to the mass and momentum balance variables of the macrocopic model up to specific orders in the scaling parameter.  

Based on the present contributions, in the sequel (part II \cite{simonis2023hlbmPartII}) the homogenized NSE are consistently approximated by deriving a HLBM discretization of the homogenized BGK Boltzmann equation (HBGKBE) (see Figure~\ref{fig:limitCons}). 
Therein, a top-down derivation of HLBMs is provided, based on the limit consistent discretizations \cite{simonis2022limit} of Boltzmann-like equations with simplified collision. 
We thus construct homogenized lattice Boltzmann equations (HLBEs) that are second and first order consistent towards the pressure and the velocity of the HNSE, respectively. 

\begin{figure}[ht!]
\centering
\includegraphics[scale=1]{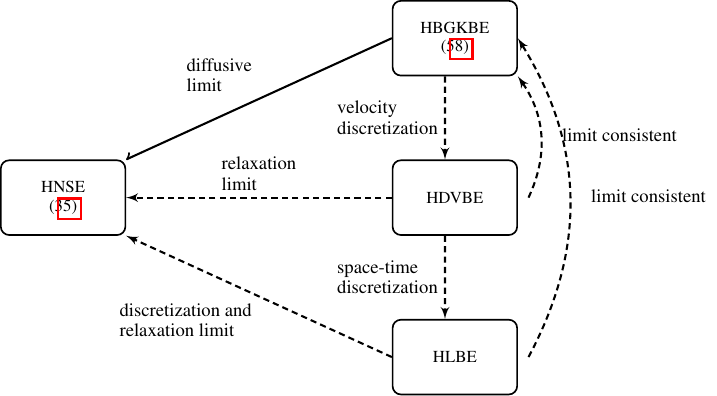}
\caption{Limit consistent derivation of HLBM. Limits considered in the present work are drawn with solid lines. Derivations considered in the sequel \protect\cite{simonis2023hlbmPartII} are dashed. The following abbreviations are used: HNSE (homogenized Navier--Stokes equations), HBGKBE (homogenized BGK Boltzmann equation), HDVBE (homogenized discrete velocity BGK Boltzmann equation), HLBE (homogenized lattice Boltzmann equation).}
\label{fig:limitCons}
\end{figure}

This work is structured as follows. 
In Section~\ref{sec:mathModel}, we summarize the geometric setup and the mathematical model based on homogenization of the stationary and the nonstationary NSE. 
Further, its physical interpretation is discussed. 
In Section~\ref{sec:kinMod}, the HBGKBE is constructed as kinetic model based on a porosity modified equilibrium. 
Convergence of the zeroth and first order moments towards variables which obey the mass and momentum balance equations of the HNSE is formally proven. 
Finally, in Section~\ref{sec:conclusion} we critically assess the present work, suggest follow-up studies and conclude the manuscript. 

\section{Mathematical Model}\label{sec:mathModel}
\subsection{Geometric setup}\label{subsec:geometric}
\begin{figure}[ht!]
\centerline{	
	\subfloat[Subvolume of the porous media]{
		\includegraphics[scale=1]{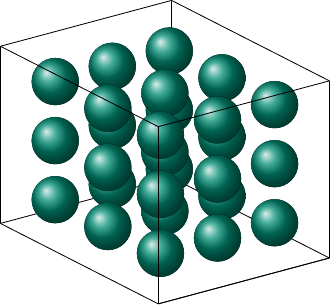}
		\label{fig:porousStructure} 
	}
	\subfloat[The \(i\)th cell]{
		\includegraphics[scale=0.7]{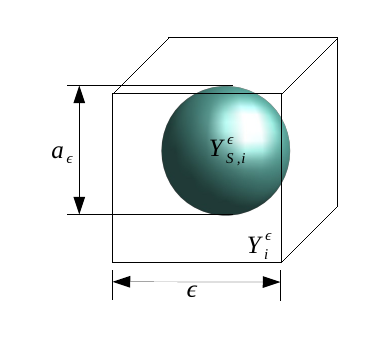}
		\label{fig:unitCell3d}
	}
}
\caption{Illustrations of the geometric model of a porous structure in \(d=3\) dimensions. 
The \(i\)th cell is denoted with \(Y_{i}^{\epsilon}\) containing a spherical matrix obstacle \(Y_{S,i}^{\epsilon}\) with radius \(a_{\epsilon}\). 
Each cell is cubic with side length \(\epsilon\).}
\end{figure}
Unless stated otherwise, \(C,C_{n} >0\) are constants, where \(n\in \mathbb{N}_{0}\). 
We model geometrically the flow through porous media via placing regularly arranged obstacles of equal size in the fluid domain \cite{allaire1989homogeneisation,allaire1991homogenization,allaire1991homogenizationNSE,allaire2010homogenizationLec2}. 
Let the domain \(\Omega \subseteq \mathbb{R}^{d}\), \(d\geq 2\) be defined as an open, bounded, and connected set. 
The boundary \(\partial\Omega\) is assumed to be smooth of class \(C^{1}\). 
The domain \(\Omega\) is covered with a regular mesh of period \(\epsilon > 0\) (see Figure~\ref{fig:porousStructure}), prescribing the cells \(Y_{i}^{\epsilon} = \left( 0, \epsilon\right)^{d}\), for \(1 \leq i \leq N\left(\epsilon\right)\) (see Figure~\ref{fig:unitCell3d}), where 
\begin{align}		
N\left(\epsilon\right) = \vert \Omega \vert \epsilon^{-d} \left( 1 + \mathcal{O}\left(1\right)\right) 
\end{align}
counts their number.  
Each cubical cell contains a solid spherical obstacle \(Y_{S,i}^{\epsilon}\) of diameter \(a_{\epsilon}\) located in its center and a complementary part filled with fluid 
\begin{align}
Y_{F,i}^{\epsilon} = Y_{i}^{\epsilon} \setminus Y_{S,i}^{\epsilon}.
\end{align} 
The overall fluid void is thus obtained via removal of the collective solid matrix, i.e.\ 
\begin{equation}\label{eq:microSolidRemoval}
	\Omega_{\epsilon} = \Omega \setminus \bigcup\limits_{i=1}^{N\left(\epsilon\right)} Y_{S,{i}}^{\epsilon} .
\end{equation}
Further, we assume that all obstacles are similar to a model obstacle \(Y^{m}_{S}\) of size \(a_{\epsilon}\).

Under the assumption that the obstacle diameter is much smaller than the cell length, i.e.\ \(a_{\epsilon} \ll \epsilon\) or equivalently 
\begin{equation}
\lim\limits_{\epsilon \searrow 0} \frac{a_{\epsilon} }{ \epsilon} = 0,
\end{equation} 
we introduce a notion of respective size for the obstacle, by defining the ratio
\begin{equation}\label{eq:ratio}
	\sigma_{\epsilon} = 
	\begin{cases}
		\left( \frac{\epsilon^{d}}{a_{\epsilon}^{d-2}} \right)^{\frac{1}{2}} \quad &\text{for } d \geq 3, \\
		\epsilon \left\vert \log \left( \frac{a_{\epsilon}}{\epsilon} \right) \right\vert^{\frac{1}{2}} \quad & \text{for } d = 2.
	\end{cases}  
\end{equation} 
\begin{proposition}
For a critical size \(a_{\epsilon} = a_{\epsilon}^{\mathrm{crit}}\) defined as
\begin{equation}\label{eq:critHoleSize}
a_{\epsilon}^{\mathrm{crit}} = 
\begin{cases}
C_{0} \epsilon^{\frac{d}{d-2}} \quad & \text{for } d \geq 3 , \\
e^{- \frac{C_{0}}{\epsilon^{2}}} \quad & \text{for } d=2 ,
\end{cases}
\end{equation} 
with \(0 < C_{0} < \infty\), the ratio \(\sigma_{\epsilon}\) reaches the nonnegative constant homogenization limit
\begin{equation}
\lim\limits_{\epsilon \searrow 0} \sigma_{\epsilon} = 
\begin{cases} 
\left( C_{0} \right)^{\frac{{2-d}}{2}} \quad & \text{for } d\geq 3, \\
\left( C_{0} \right)^{\frac{1}{2} } \quad & \text{for } d=2.
\end{cases}
\end{equation} 
\end{proposition}
\begin{proof}
Substituting \eqref{eq:critHoleSize} into \eqref{eq:ratio} completes the proof. 
\end{proof}
Below in Section~\ref{section:homoNS}, it will be shown that, for \(\epsilon\searrow 0\), large obstacles correspond to the limit \(\sigma_\epsilon\rightarrow\infty\), small obstacles to the limit \(\sigma_\epsilon\rightarrow 0\), and critical obstacles to the limit \(\sigma_\epsilon\) with \(0<\sigma<\infty\).

\subsection{Homogenized stationary Navier\textendash Stokes equations}\label{section:homoNS}
\noindent
In the case of independence from time, the incompressible fluid flow in \(\Omega_{\epsilon}\) is modeled by the stationary nonlinear NSE\(_{\epsilon}\)
\begin{equation}\label{eq:microNSE}
	\begin{cases}
	\bm{u}_{\epsilon} \cdot \bm{\nabla}_{\bm{x}} \bm{u}_{\epsilon} - \nu \bm{\Delta}_{\bm{x}} \bm{u}_{\epsilon} = \bm{F} - \bm{\nabla}_{\bm{x}} p_{\epsilon} 
		\quad &\text{in } \Omega_{\epsilon}, \\
		\mathrm{div}_{\bm{x}} \bm{u}_{\epsilon} = 0 
		\quad &\text{in } \Omega_{\epsilon}, \\
		\bm{u}_{\epsilon} = \bm{0} 
		\quad &\text{on } \partial\Omega_{\epsilon}, 
	\end{cases}
\end{equation}  
where \(\bm{u}_{\epsilon}\colon \Omega_{\epsilon} \to \mathbb{R}^{d}\) denotes the velocity field, \(p_{\epsilon}\colon \Omega_{\epsilon} \to \mathbb{R}\) is the scalar-valued pressure, \(\bm{F} \in L^{2}\left(\Omega\right)^{d}\) defines a given force, and \(\nu > 0\) is a constant viscosity. 
Additionally, to comply with \eqref{eq:microSolidRemoval}, we define the extension to \(\Omega\) of a pair of solutions \(\left( \bm{u}_{\epsilon}, p_{\epsilon}\right)\) of \eqref{eq:microNSE} as
\begin{equation}\label{eq:microNSE_completion}
	\left( \tilde{\bm{u}}_{\epsilon}, \tilde{p}_{\epsilon}  \right) =
	\begin{cases}
	 \left( \bm{u}_{\epsilon} ,  p_{\epsilon} \right)  \quad & \text{in } \Omega_{\epsilon} ,  \\
	 \left( \bm{0}, \frac{1}{\vert C_{i}^{\epsilon}\vert} \int_{C_{i}^{\epsilon}} p_{\epsilon} \,\mathrm{d}\bm{x} \right) \quad & \text{in each obstacle } Y_{S,{i}}^{\epsilon},
	\end{cases}
\end{equation} 
where \(C_{i}^{\epsilon}\) denotes a control volume containing \(Y_{S,{i}}^{\epsilon}\) \cite{allaire1991homogenizationNSE}. 
Heuristically, this means setting inside the obstacle zero velocity and the average value of the pressure field in its immediate proximity.

Based on the above definitions, Allaire~\cite{allaire1991homogenizationNSE} (see Corollary~1.4 therein) proved homogenization limits for different obstacle sizes expressed in the ratio \eqref{eq:ratio}. 
The results of homogenizing \eqref{eq:microNSE} are summarized in the following statements which are recalled without proof. 
Let the index \(\cdot_{0}\) of a function space denote the classical vanishing trace operator, e.g.\ for \(X\subseteq \mathbb{R}^{d}\) let 
\begin{align}
H^{1}_{0} (X) = \left\{ f \in H^{1}(X) \;\left\vert\; f\vert_{\partial X} = 0 \right. \right\} .
\end{align}

\begin{theorem}\label{thmStatNSE}
According to the scaling of the obstacle size, we distinguish between three homogenization limits.
\begin{enumerate}
\item[(\underline{i})] If the obstacles are too small, i.e.\ \(\lim_{\epsilon\searrow 0} \sigma_{\epsilon} = + \infty\), then \(\left(\tilde{\bm{u}}_{\epsilon}, \tilde{p}_{\epsilon} \right)\) converges strongly in \(H_{0}^{1} \left( \Omega \right)^{d} \times L^{2}\left(\Omega\right) / \mathbb{R}\) to \(\left( \bm{u}, p \right)\), a solution of the stationary nonlinear NSE
\begin{equation} \label{eq:obstacleTooSmall}
\begin{cases}
	\bm{u} \cdot \bm{\nabla}_{\bm{x}} \bm{u} - \nu \bm{\Delta}_{\bm{x}} \bm{u} = \bm{F} - \bm{\nabla}_{\bm{x}} p  \quad &\textit{in } \Omega , \\
	\mathrm{div}_{\bm{x}} \bm{u} = 0 			 \quad &\textit{in } \Omega , \\
	\bm{u} = \bm{0} 						 \quad &\textit{on } \partial\Omega .
\end{cases} \quad 
\end{equation} 
\item[(\underline{ii})] If the obstacles have a critical size, i.e.\ \(\lim_{\epsilon\searrow 0} \sigma_{\epsilon} = \sigma > 0\), then \(\left( \tilde{\bm{u}}_{\epsilon}, \tilde{p}_{\epsilon}\right)\) converges weakly in \(H^{1}_{0} \left(\Omega \right)^{d} \times L^{2} \left(\Omega\right) / \mathbb{R}\) to \(\left( \bm{u}, p \right)\), a solution of the stationary nonlinear BL  
\begin{equation} \label{eq:obstacleCritical}
\begin{cases}
	\bm{u}\cdot \bm{\nabla}_{\bm{x}} \bm{u} - \nu \bm{\Delta}_{\bm{x}} \bm{u} + \frac{\nu}{\sigma^{2}} \mathbf{M} \bm{u} = \bm{F} - \bm{\nabla}_{\bm{x}} p  \quad &\textit{in } \Omega , \\
	\mathrm{div}_{\bm{x}} \bm{u} = 0 										   \quad &\textit{in } \Omega , \\
	\bm{u} = \bm{0} 													   \quad &\textit{on } \partial\Omega .
\end{cases} \quad 
\end{equation} 
\item[(\underline{iii})] If the obstacles are too big, i.e.\ \(\lim_{\epsilon \searrow 0} \sigma_{\epsilon} = 0\), then the rescaled solution \(\left( \frac{\tilde{\bm{u}}_{\epsilon}}{\sigma_{\epsilon}^{2}}, \tilde{p}_{\epsilon}\right)\) converges strongly in 
\(H^{1}_{\mathrm{div}}\left( \Omega \right)^{d} \times L^{2}\left(\Omega\right) / \mathbb{R}\) 
%\(L^{2}\left( \Omega\right)^{d} \times L^{2}\left(\Omega\right) / \mathbb{R}\) 
to \(\left( \bm{u}, p \right)\), the unique solution of the DL
\begin{equation} \label{eq:darcyLaw1}
\begin{cases}
	\bm{u} =  \frac{1}{\nu}  \mathbf{M}^{-1}\left( \bm{F} - \bm{\nabla}_{\bm{x}} p \right)  	 \quad &\textit{in } \Omega , \\
	\mathrm{div}_{\bm{x}} \bm{u} = 0 										 \quad &\textit{in } \Omega , \\
	\bm{u}\cdot\bm{n} = \bm{0} 													 \quad &\textit{on } \partial\Omega ,
\end{cases} 
\end{equation} 
where \(\bm{n}\in \mathbb{R}^{d}\) is the outward pointing normal vector.
\end{enumerate}
In all three regimes (\underline{i}--\underline{iii}), \(\mathbf{M}\) is a \(d \times d\) symmetric matrix, which depends only on the model obstacle \(Y^{m}_{S}\). 
\end{theorem}
\begin{proof}
Proofs for all cases are provided in \cite{allaire1991homogenizationNSE}. 
\end{proof}

The porosity matrix \(\mathbf{M}\), which inversely represents a permeability tensor (see below Theorem~\ref{thm:lowVolumeFrac}), is computable via a model problem defined locally around \(Y^{m}_{S}\) (see Proposition~1.2 in \cite{allaire1991homogenizationNSE} and Proposition~1.3.2 in \cite{allaire2010homogenizationLec2}). 
The following result, obtained from merging Proposition~1.2 in \cite{allaire1991homogenizationNSE} and Proposition~1.3.2 in \cite{allaire2010homogenizationLec2}, unfolds the computation of \(\mathbf{M}\). 

\begin{proposition}\label{prop:locModProbSNSE}
Let \(\left\{ \bm{e}_{k} \right\}_{1\leq k\leq d}\) denote the unit basis of \(\mathbb{R}^{d}\). 
Hence, the local model problem is defined for each \(k\) as 
\begin{equation}\label{eq:localModelProb}
\begin{cases}
\bm{\nabla}_{\bm{x}} q_{k} - \bm{\Delta}_{\bm{x}} \bm{w}_{k} = \bm{0} \quad & \text{in } \mathbb{R}^{d} \setminus Y^{m}_{S} , \\
\mathrm{div}_{\bm{x}}  \bm{w}_{k} = \bm{0} \quad & \text{in } \mathbb{R}^{d} \setminus Y^{m}_{S} , \\
\bm{w}_{k} = \bm{0} \quad & \text{on } \partial Y^{m}_{S} , \\
\bm{w}_{k}  
	\begin{cases}
	\to \bm{e}_{k} \quad & \text{for } d \geq 3  \\
	\sim \bm{e}_{k} \log \left( \vert \bm{x} \vert \right) \quad & \text{for } d = 2 \\
	\end{cases}
\quad & \text{as } \vert \bm{x} \vert \to \infty .
\end{cases}
\end{equation}  
The matrix \(\mathbf{M}\) is then assembled through 
\begin{equation} \label{eq:porosityMatrix}
 \mathbf{M}  = 
\begin{cases}
\begin{bmatrix}
\int_{\mathbb{R}^{d} \setminus Y_{S}} \bm{\nabla}_{\bm{x}} \bm{w}_{k} \cdot \bm{\nabla}_{\bm{x}} \bm{w}_{j} \,\mathrm{d}\bm{x}
\end{bmatrix}
_{1\leq j, k \leq d} \quad & \text{for } d \geq 3 , \\
4 \pi \mathbf{I}_{d} \quad & \text{for } d = 2 .
\end{cases}
\end{equation} 
\end{proposition}

\begin{remark}
Note that the standard derivation of the DL uses the assumption that the obstacle size \(a_{\epsilon} = \mathcal{O} (\epsilon) \). 
Presently, so far we have assumed a smaller obstacle size. 
Hence, the typical permeability tensor (often referred to as \(\mathbf{K}\)) is computed from a different model problem as the local model problem~\eqref{eq:localModelProb}. 
Allaire~\cite{allaire1991continuity} closely examines the relation of permeability tensors and porosity matrices, and states the following result. 
\end{remark}
 
Let the obstacle size be redefined as \(a_{\epsilon} \coloneqq \delta \epsilon = \mathcal{O} \left( \epsilon \right)\). 
Let \(\iota_{i}^{\epsilon}\) define a linear homeomorphism, mapping each cell to the unit cell \(Y = \left(0, 1\right)^d\) and allocating solid and fluid parts therein, \(Y_{S} = \iota_{i}^{\epsilon}\left( Y_{S,{i}}^{\epsilon}\right)\) and \(Y_{F} = \iota_{i}^{\epsilon}\left( Y_{F,{i}}^{\epsilon} \right)\), respectively. 
Hence, the unit cell \(Y\) now is split into a fluid part \(Y_{F} = Y \setminus Y_{S}\) and an obstacle \( Y_{S}\) which is of size \(\delta > 0\) due to \(\iota_{i}^{\epsilon}\) resembling a rescaling with a homothety factor of \(\epsilon^{-1}\) \cite{mikelic1994mathematical}. 
The following theorem states the outcome of the homogenization in this case. 

\begin{theorem}\label{thmStatNSECase4}
An extension \(\left( \tilde{\bm{u}}_{\epsilon}, \tilde{p}_{\epsilon}\right) \) of the solution \(\left(\bm{u}_{\epsilon}, p_{\epsilon}\right)\) of \eqref{eq:microNSE} exists, such that \(\tilde{\bm{u}}_{\epsilon}\) converges weakly in \(L^{2}\left( \Omega\right)^{d}\) to \(\bm{u}\), and \(\tilde{p}_{\epsilon}\) converges strongly in \(L^{q^{\prime}} \left( \Omega\right) / \mathbb{R}\) to \(p\), for any \(1 < q^{\prime} <\beta\), where \(\left( \bm{u},  p\right) \) is the unique solution of the DL
\begin{equation}\label{eq:darcyLaw2}
\begin{cases}
	\bm{u} =  \frac{1}{\nu} \mathbf{A} \left( \bm{F} - \bm{\nabla}_{\bm{x}} p \right)  	 \quad &\textit{in } \Omega , \\
	\mathrm{div}_{\bm{x}} \bm{u} = 0 										 \quad &\textit{in } \Omega , \\
	\bm{u} \cdot \bm{n} = \bm{0} 													 \quad &\textit{on } \partial\Omega .
\end{cases} 
\end{equation} 
In the DL \eqref{eq:darcyLaw2}, the porosity matrix \(\mathbf{A}\) is defined by 
\begin{equation} \label{eq:permeabilityTensor}
 \mathbf{A}  = \left[\;
 \int_{Y_{F}} \bm{\nabla}_{\bm{x}} \bm{v}_{k} \cdot \bm{\nabla}_{\bm{x}} \bm{v}_{j} \,\mathrm{d}\bm{x} \;\right]_{1\leq j, k \leq d}, 
\end{equation} 
where for the canonical basis vector \(\bm{e}_{k}\), \(1 \leq k \leq d\), of \(\mathbb{R}^{d}\), \(\bm{v}_{k}\) is the unique solution in \(H^{1}_{\#}\left( Y_{F}\right)^{d} \) of the unit cell problem 
\begin{equation}\label{unitCellProb}
\begin{cases}
\bm{\nabla}_{\bm{x}} p_{k} - \bm{\Delta}_{\bm{x}} \bm{v}_{k} = \bm{e}_{k} \quad & \text{in } Y_{F} , \\
\mathrm{div}_{\bm{x}} \bm{v}_{k} = 0 \quad & \text{in }  Y_{F} , \\
\bm{v}_{k} = \bm{0} \quad & \text{on } \partial \left( Y_{S}\right) , 
\end{cases}
\end{equation} 
where \(H^{1}_{\#}\left( Y_{F}\right)\) denotes the Sobolev space of \(Y_{F}\)-periodic functions in \(H^{1}\left(Y_{F}\right)\). 
\end{theorem} 
\begin{proof}
The theorem is a special case of Theorem~1.2.5 in \cite{allaire2010homogenizationLec2} which restates the rigorous result of Mikeli\'{c}~\cite{mikelic1991homogenization}. 
Hence, we choose the specific constants in Theorem~1.2.5 from \cite{allaire2010homogenizationLec2} as \(\gamma = 4\) and \(\beta >1\) which completes the proof. 
\end{proof}

Further, the continuity in the low volume fraction limit (\(\delta \searrow 0\)) is verified through the following theorem, which links the permeability tensor \(\mathbf{A}\) \eqref{eq:permeabilityTensor} in the DL \eqref{eq:darcyLaw2} to the porosity matrix \(\mathbf{M}\) \eqref{eq:porosityMatrix} in the DL \eqref{eq:darcyLaw1}. 
\begin{theorem}\label{thm:lowVolumeFrac}
Let \(\left(   p_{k}, \bm{v}_{k} \right)\) be the unique solution of the unit cell problem. 
Rescaling it, for \(\bm{x}\in \delta^{-1} \left( Y \setminus Y_{S}\right)\), we can define 
\begin{align}
\bm{v}_{k}^{\delta} \left( \bm{x} \right) & = \delta^{d-2} \bm{v}_{k}\left( \delta \bm{x}\right), \\
p_{k}^{\delta}\left( \bm{x}\right) & = \delta^{d-1} p_{k}\left( \delta \bm{x}\right). 
\end{align} 
Further, let \(\left( q_{i}, \bm{w}_{i}\right) \) be the unique solution of the local model problem. 
Then \(\left( p_{k}^{\delta}, \bm{v}_{k}^{\delta} \right)\) converges weakly to
\begin{equation}
\sum\limits_{i=1}^{d} \left( \bm{e}_{i}^{\mathrm{T}} \mathbf{M}^{-1} \bm{e}_{k} \right) \left( q_{i}, \bm{w}_{i}\right)
\end{equation} 
in \( \left[ L_{\mathrm{loc}}^{2} \left( \mathbb{R}^{d} \setminus Y_{S} \right) / \mathbb{R} \right] \times \left[ H_{\mathrm{loc}}^{1} \left( \mathbb{R}^{d} \setminus Y_{S} \right) \right]^{d}\). 
Additionally, the low volume fraction limit for the permeability tensor is given as 
\begin{equation}
\begin{cases}
\lim\limits_{\delta \searrow 0} \delta^{d-2}\mathbf{A}\left(\delta\right) =\mathbf{M}^{-1}, \quad &\textit{for } d\geq 3, \\
\lim\limits_{\delta \searrow 0} \frac{1}{\vert \log \delta \vert } \mathbf{A}\left( \delta\right) = \mathbf{M}^{-1}, \quad &\textit{for } d= 2.
\end{cases}
\end{equation} 
\end{theorem}
\begin{proof}
The theorem is proven by Allaire~\cite{allaire1991continuity} (see Theorem~3.1 therein). 
\end{proof}

\begin{remark}
Thus, for the complete range \(a_{\epsilon} \leq \mathcal{O}\left( \epsilon\right) \), the homogenized stationary equations~\eqref{eq:obstacleTooSmall},~\eqref{eq:obstacleCritical},~\eqref{eq:darcyLaw1} and~\eqref{eq:darcyLaw2} are obtained as limits of~\eqref{eq:microNSE} and~\eqref{eq:microNSE_completion}.
Specifically, the case \(a_{\epsilon} < \mathcal{O}\left( \epsilon\right) \) is covered in Theorem~\ref{thmStatNSE}~via (i--iii), and the complementary case~(iv) \(a_{\epsilon} = \mathcal{O}\left( \epsilon\right) \), via Theorem~\ref{thm:lowVolumeFrac}.
\end{remark}

\subsection{Homogenized nonstationary Navier--Stokes equations}
Let the domain be defined as above and \(d\in \left\{ 2, 3 \right\}\). 
The incompressible fluid flow, now being dependent on time \(t \in I = \left( 0, T\right)\), is governed by the nonstationary nonlinear NSE\(_{\epsilon}\)
\begin{equation}\label{eq:microEvoNSE}
	\begin{cases}
		\partial_{t} \bm{u}_{\epsilon} + \epsilon^{4} \bm{u}_{\epsilon} \cdot \bm{\nabla}_{\bm{x}} \bm{u}_{\epsilon} - \epsilon^{2} \nu \bm{\Delta}_{\bm{x}} \bm{u}_{\epsilon} = \bm{F} - \bm{\nabla}_{\bm{x}} p_{\epsilon}
		\quad &\text{in } \Omega_{\epsilon} \times I, \\
		\mathrm{div}_{\bm{x}} \bm{u}_{\epsilon} = 0 
		\quad &\text{in } \Omega_{\epsilon} \times I, \\
		\bm{u}_{\epsilon} \vert_{t=0} = \bm{u}_{0, \epsilon}
		\quad &\text{in } \Omega_{\epsilon}, \\
		\bm{u}_{\epsilon} = \bm{0} 
		\quad &\text{on } \partial\Omega_{\epsilon} \times I, 
	\end{cases}
\end{equation}  
where \(\bm{u}_{\epsilon}\colon \Omega_{\epsilon} \times I \to \mathbb{R}^{d}\) denotes the velocity field, \(p_{\epsilon}\colon \Omega_{\epsilon} \times I \to \mathbb{R}\) is the scalar-valued pressure, \(\bm{F} \in L^{2} (I; L^{2} \left( \Omega_{\epsilon}\right)^{d})\) defines a given force, \(\nu > 0\) is a constant viscosity, and \(\partial \Omega_{\epsilon}\) is supposed to be sufficiently regular. 
\begin{remark}
Note that the individual terms of \eqref{eq:microEvoNSE} are properly rescaled by prefactors of \(\epsilon\) to ensure a non-vanishing limit velocity \cite{allaire2010homogenizationLec2}. 
\end{remark} 
Further, following \cite{feireisl2016homogenization}, let 
\begin{equation}\label{eq:initField}
\begin{cases}
\bm{u}_{0, \epsilon}  \in L^{2}\left( \Omega_{\epsilon} \right)^{d}, \\
\mathrm{div}_{\bm{x}} \bm{u}_{0,\epsilon} = 0 \quad & \text{in } \Omega_{\epsilon}, \\
\bm{u}_{0,\epsilon} \cdot \bm{n} = 0 \quad & \text{on } \partial \Omega_{\epsilon} .
\end{cases}
\end{equation}  
%where \(\bm{n}\in \mathbb{R}^{d}\) is the outward pointing normal vector. 
In this configuration, at least one weak solution to \eqref{eq:microEvoNSE} exists \cite{feireisl2016homogenization}, which is obtained in \( \bm{u}_{\epsilon} \in L^{2} ( I; H^{1} ( \Omega_{\epsilon} )^{d} ) \) and \( p_{\epsilon} \in   H^{-1} ( I; L_{0}^{2} ( \Omega_{\epsilon} ) ) \), respectively \cite{mikelic1994mathematical}. 
To formulate the nonstationary version of Theorem~\ref{thmStatNSE}, the works of Feireisl \textit{et al.}~\cite{feireisl2016homogenization}, Allaire~\cite{allaire1992homogenization}, and Mikeli\'{c}~\cite{mikelic1991homogenization}
serve as a basis. 
Since only parts of the limit cases have been proven yet, we formulate a conjecture for the nonstationary case below.

\begin{definition}\label{def:timeDepUCPNSE} 
Let 
\begin{align}
\begin{cases}
\partial_{t} \bm{w}^{j} - \nu \bm{\Delta}_{\bm{x}} \bm{w}^{j} + \bm{\nabla}_{\bm{x}} \pi^{j} = \bm{0} \quad &\text{in } Y_{F} \times I , \\
\mathrm{div}_{\bm{x}} \bm{w}^{j} = 0 \quad &\text{in } Y_{F} \times I , \\
\bm{w}^{j}\vert_{t=0} = \bm{e}^{j} \quad &\text{in } Y_{F}, \\
\bm{w}^{j} = \bm{0} \quad &\text{on } (\partial Y_{S} \backslash \partial Y ) \times I 
\end{cases}
\end{align}
define a time-dependent unit cell problem \cite{mikelic1994mathematical}, where \(\bm{w}^{j}\) is \(H^{1}(Y)\)-periodic and \(\pi^{j}\) is \(L^{2}(Y)\)-periodic, component-wise. 
The matrix \(\tilde{\mathbf{A}} ( t )\) is then assembled through 
\begin{align}
\tilde{A}_{ij} \left( t \right)  = \frac{1}{\vert Y \vert} \int_{Y_{F}} w_{j}^{i} \left( \bm{y}, t \right) \,\mathrm{d} \bm{y}, 
\end{align}
for \(1\leq i,j \leq d\). 
\end{definition}

\begin{conjecture}\label{thmEvoNSE}
Let \(\left(\tilde{\bm{u}}_{\epsilon}, \tilde{p}_{\epsilon}\right)\) be a weak solution to \eqref{eq:microEvoNSE}. 
Assume that \(\lim_{\epsilon\searrow 0} \bm{u}_{0,\epsilon} = \bm{u}_{0} \) weakly in \(L^{2}(\Omega )^{d}\). 
According to the scaling regimes of the obstacle size, we distinguish between the following homogenization limits.
\begin{enumerate}%[label=(\roman*)]
\item[(i)] If the obstacles are too small, i.e.\ \(\lim_{\epsilon\searrow 0} \sigma_{\epsilon} = + \infty\), then \(\left(\tilde{\bm{u}}_{\epsilon}, \tilde{p}_{\epsilon} \right)\) converges to \(\left( \bm{u}, p \right)\), a solution of the nonstationary nonlinear NSE
\begin{equation}\label{eq:conjectureNSE}
\begin{cases}
	\partial_{t} \bm{u} + \bm{u} \cdot \bm{\nabla}_{\bm{x}} \bm{u} - \nu \bm{\Delta}_{\bm{x}} \bm{u} = \bm{F} - \bm{\nabla}_{\bm{x}} p  \quad &\textit{in } \Omega \times I , \\
	\mathrm{div}_{\bm{x}} \bm{u} = 0 			 \quad &\textit{in } \Omega \times I, \\
    \bm{u}\vert_{t=0} = \bm{u}_{0} \quad & \textit{in } \Omega, \\
	\bm{u} = \bm{0} 						 \quad &\textit{on } \partial\Omega \times I  .
\end{cases} \quad 
\end{equation} 
\item[(ii)] If the obstacles have a critical size, i.e.\ \(\lim_{\epsilon\searrow 0} \sigma_{\epsilon} = \sigma > 0\), then \(\left( \tilde{\bm{u}}_{\epsilon}, \tilde{p}_{\epsilon}\right)\) converges in \(L^{2}(\Omega \times I  )\) and weakly in \(L^{2}(I; W_{0}^{1,2}(\Omega))\) to \(\left( \bm{u}, p \right)\), respectively, a solution of the nonstationary nonlinear BL
\begin{equation}\label{eq:conjectureBL}
\begin{cases}
	\partial_{t} \bm{u} +  \bm{u} \cdot \bm{\nabla}_{\bm{x}} \bm{u} - \nu \bm{\Delta}_{\bm{x}} \bm{u} + \frac{\nu}{\sigma^{2}} \mathbf{M} \bm{u} = \bm{F} - \bm{\nabla}_{\bm{x}} p  \quad &\textit{in } \Omega \times I  , \\
	\mathrm{div}_{\bm{x}} \bm{u} = 0 										   \quad &\textit{in } \Omega \times I , \\
     \bm{u}\vert_{t=0} = \bm{u}_{0} \quad & \textit{in } \Omega, \\
	\bm{u} = \bm{0} 													   \quad &\textit{on } \partial\Omega \times I  .
\end{cases} \quad 
\end{equation} 
\item[(iii)] If the obstacles are smaller than \(\mathcal{O}\left(\epsilon\right) \), but exceed the critical size, such that \(\lim_{\epsilon \searrow 0} \sigma_{\epsilon} = 0\), then a suitably rescaled version of \(\left( \tilde{\bm{u}}_{\epsilon}, \tilde{p}_{\epsilon}\right)\) converges to \(\left( \bm{u}, p \right)\), the unique solution of the time-dependent DL
\begin{equation} \label{eq:darcyLawTimeDep}
\begin{cases}
\partial_{t} \bm{u}  +	\nu \mathbf{M} \bm{u} =   \bm{F} - \bm{\nabla}_{\bm{x}} p   	 \quad &\textit{in } \Omega \times I, \\
	\mathrm{div}_{\bm{x}} \bm{u} = 0 										 \quad &\textit{in } \Omega  \times I , \\
     \bm{u}\vert_{t=0} = \bm{u}_{0} \quad & \textit{in } \Omega, \\
	\bm{u}\cdot \bm{n} = \bm{0} 													 \quad &\textit{on } \partial\Omega  \times I .
\end{cases} \quad 
\end{equation} 
\item[(iv)] %Mikeli\'{c} \cite[Theorem 1.2 (let \(\beta = 4\))]{mikelic1994mathematical}: 
If the obstacles are of size \(\mathcal{O} ( \epsilon )\), then the rescaled solution \( ( \epsilon^{2} \tilde{\bm{u}}_{\epsilon}, \partial_{t} \tilde{p}_{\epsilon} )\) converges in \( L^{2} ( I; \Omega )^{d} \) and weakly in \( H^{-1} ( I; L_{0}^{2} ( \Omega ) )\), respectively to \(\left( \bm{u}, p \right)\), the unique solution of the DL with memory
\begin{equation} \label{eq:darcyLawMemory}
\begin{cases}
\nu \bm{u} - \tilde{\mathbf{A}}\left(t\right) \bm{u}_{0} = \int_{0}^{t} \tilde{\mathbf{A}}\left(t-s\right) \left[\bm{F}\left(s\right) - \bm{\nabla}_{\bm{x}} p\left( s\right) \right] \,\mathrm{d}s  	 \quad &\textit{in } \Omega  \times I , \\
	\mathrm{div}_{\bm{x}} \bm{u} = 0 										 \quad &\textit{in } \Omega  \times I , \\
%     \bm{u}\vert_{t=0} = \bm{u}_{0} \quad & \textit{in } \Omega, \\
	\bm{u} \cdot \bm{n} = \bm{0} 													 \quad &\textit{on } \partial\Omega  \times I .
\end{cases} \quad 
\end{equation} 
Further, if the flow stabilizes after a finite period of time, the DL with memory \eqref{eq:darcyLawMemory} contracts for \(t\to \infty\) to the classical DL \eqref{eq:darcyLaw2} with 
\begin{equation}
A_{ij} = \int_{0}^{\infty} \tilde{A}_{ij} \left( t\right) \,\mathrm{d} t ,
\end{equation} 
for \(1\leq i,j, \leq d\). 
\end{enumerate}
In the regimes (i-iii), \(\mathbf{M}\) is the same \(d \times d\) symmetric matrix as in Proposition~\ref{prop:locModProbSNSE} and depends only on the model obstacle \(Y_{S}^{m}\).  
In case of (iv), \(\tilde{\mathbf{A}}\left( t \right) \) is constructed from Definition \ref{def:timeDepUCPNSE}. 
\end{conjecture}

\begin{proof}[Proof of cases (ii) and (iv)]
 In contrast to the stationary case (see Theorem~\ref{thmStatNSE}, Theorem~\ref{thmStatNSECase4} and Theorem~\ref{thm:lowVolumeFrac}) only parts of the homogenization limits in Conjecture \ref{thmEvoNSE} have been proven yet. 
 In addition, to the knowledge of the authors, none of the interconnections between individual cases (i--iv) have been established yet. 
Hence we recall the available proofs only. %\textit{Case (ii).} 
Feireisl \textit{et al.}~\cite{feireisl2016homogenization} proved case (ii), where obstacles of critical size smaller than \(\mathcal{O}\left( \epsilon\right)\) are considered. 
However a differing methodology to the one used by Allaire is applied to rigorously pass to the limit equations. 
Via the techniques used in \cite{feireisl2016homogenization}, the above assumptions on the shape and location of the obstacles can be loosened. 
The resulting homogenized equations however, are a BL which is similar to the one obtained in the framework introduced above. 
The case (iv) above is a special case of the derivations in Mikeli\'{c}~\cite{mikelic1994mathematical} (see Theorem~1.2 therein, with \(\beta = 4\)) and is thus rigorously proven. 
\end{proof}

\begin{remark}
Cases (i) and (iii) are based on the conclusive evidence in the literature (see e.g.\ \cite{allaire1991homogenization,feireisl2016homogenization}) for the resulting PDEs when homogenizing the nonlinear nonstationary NSE~\eqref{eq:microEvoNSE}. 
For a rigorous proof, starting from the homogenization limits of the nonstationary Stokes equations established in \cite{allaire1992homogenization} could be promising, since, as stated in \cite{allaire1991homogenization}, the inclusion of a nonlinear advective term to the Stokes equations resembles a compact perturbation of the \(\epsilon\)-dependent stationary nonlinear NSE~\eqref{eq:microNSE}. 
It is also notable that for case (iii), a proof for the homogenization of the nonstationary Stokes equations (without the nonlinear advective term) is given in \cite{allaire1992homogenization}. 
Concerning the low volume fraction limit which connects cases (iv) to (iii), the memory effective terms of the DL with memory \eqref{eq:darcyLawMemory} might induce the time-dependency in the time-dependent DL \eqref{eq:darcyLawTimeDep}. 
Hence, an import of stationary effects and an additional solving for time-dependent eigenvalue problems in respective cell spaces \cite{allaire1991continuity} might be insightful. 
\end{remark}

\subsection{Applicability of the homogenized model}\label{subsec:applicability}
\begin{assumption}\label{ass:thmAssump}
To establish a connection to experimentally conforming model equations, we make the following assumptions:
\begin{enumerate} 
			\item 
			According to \cite{bear1972dynamics,mikelic1994mathematical}, the stabilization of the DL with memory \eqref{eq:darcyLawMemory} toward the classical DL \eqref{eq:darcyLaw2} is understood to happen in a short period of time.
				  Hence, we assume a stabilized flow in case of obstacle sizes which obey Conjecture~\ref{thmEvoNSE} case (iv), i.e.\ the homogenization limit is constituted by an ordinary DL. 
				  Similarly, we assume stabilization for case (iii). 
			      Typically this involves adding Brinkman terms (diffusion) or other necessary features to the DL in case (iv). 
			      Though these artificial features are effective in the void and within the porous--void interface, they are contracted to zero within the porous media under the necessary local assumptions of highly viscous and stabilized (stationary) flow. 
			\item The porosity is determined to be constant in \(\Omega_{T}\). 
			\item The medium is isotropic, which results in regular symmetric, hence diagonal or diagonalizable matrices \(\mathbf{M}\) and \(\mathbf{A}\). 
				  Further, we may thus reduce the matrix \(\mathbf{A}\) or \(\mathbf{M}^{-1}\) to its only eigenvalue, which yields a scalar multiplication. 
				  Below we assume this simplification and unless stated otherwise, denote the single eigenvalue of \(\mathbf{A}\) with \(A\). 
		\end{enumerate}
\end{assumption}
\begin{remark}
Assumption~\ref{ass:thmAssump} supports the commonly formulated Brinkman equation \cite{hornung1997homogenization,nield2017convection}, which is constituted by a classical DL plus a diffusion term. 
Neglecting the time-dependency in the BL derived above as well as its inertial terms, results in a simplified equation which solely respects diffusion. 
To match the porous\textendash void interface, Spaid and Phelan~\cite{spaid1997lattice} used such a Brinkman equation as a stationary limit for their simulations. 
A note in \cite{spaid1997lattice} additionally states that far from the interface, and within the porous domain region, the governing equation reduces again, to the classical DL \cite{spaid1997lattice}. 
It should however be noted that along the stationary limit, the nonstationary solution to the method in \cite{spaid1997lattice} is rather a BL as presently formulated, which was not further examined therein. 
\end{remark}

For \(d=3\) (see Figure~\ref{fig:unitCell3d}), \(\sigma_{\epsilon}\) describes the square root of the ratio of the cell volume to the obstacle diameter 
\begin{align}\label{eq:sigma}
\sigma_{\epsilon} = \left( \frac{\epsilon^{3}}{a_{\epsilon}} \right)^{\frac{1}{2}}. 
\end{align}
We use the classical notion of porosity \cite{hornung1997homogenization} to assess the above framework in terms of applicability. 

\begin{proposition}\label{prop:porosity}
For Conjecture~\ref{thmEvoNSE}(iv) we obtain the minimal porosity of \(\varphi \approx 0.4764\). 
\end{proposition}
\begin{proof}
We limit our analysis to polynomial ansatz up to degree 4. 
Recalling Conjecture~\ref{thmEvoNSE}, the size \(a_{\epsilon}\) of the obstacles for \(d=3\) can be distinguished in the order of \(\epsilon\) as follows. 
Let \(0<C<\epsilon\) be a constant prefactor. 
\begin{itemize}
\item[(i)] Let \(a_{\epsilon} = C \epsilon^{4} = \mathcal{O}\left( \epsilon^{4}\right)\). Then \(\sigma_{\epsilon} = \left( \frac{1}{C \epsilon}\right) ^{\frac{1}{2}} \Rightarrow \lim\limits_{\epsilon \searrow 0} \sigma_{\epsilon} = + \infty\).
\item[(ii)] Let \(a_{\epsilon} = C \epsilon^{3} = \mathcal{O} \left( \epsilon^{3}\right)\). Then \(\sigma_{\epsilon} = \left( \frac{1}{C} \right) ^{\frac{1}{2}} \Rightarrow \lim\limits_{\epsilon \searrow 0} \sigma_{\epsilon} = \sigma > 0\).
\item[(iii)] Let \(a_{\epsilon} = C \epsilon^{2} = \mathcal{O} \left( \epsilon^{2} \right)\). Then \(\sigma_{\epsilon} = \left( \frac{\epsilon}{C} \right) ^{\frac{1}{2}} \Rightarrow \lim\limits_{\epsilon \searrow 0} \sigma_{\epsilon} = 0\). 
\item[(iv)]  Let \(a_{\epsilon} = C \epsilon^{1} = \mathcal{O} \left( \epsilon^{1}\right)\). Then \(\sigma_{\epsilon} = \left( \frac{\epsilon^{2}}{C} \right)^{\frac{1}{2}} \Rightarrow \lim\limits_{\epsilon \searrow 0} \sigma_{\epsilon} = 0\). 
\end{itemize}
For the purpose of illustration, the limits of \(\sigma_{\epsilon}\) for the cases (i--iv) are plotted in Figure~\ref{fig:sigmaEps} with a fixed constant \(C=1\). 
\begin{figure}[ht!]
\centering
	\includegraphics[scale=1]{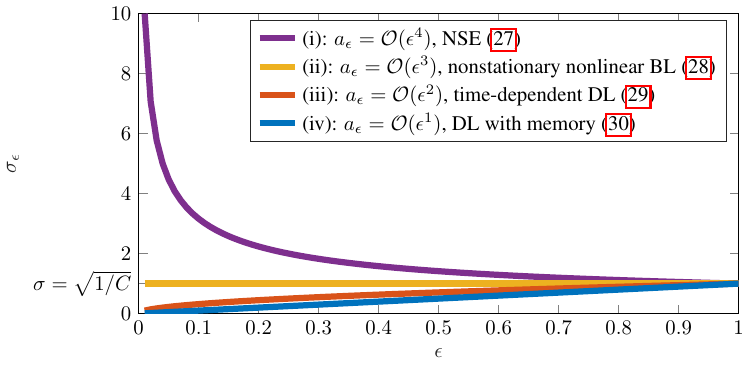}
\caption{Graph of ratio \(\sigma_{\epsilon} ( \epsilon )\) \eqref{eq:sigma} for \(d=3\) and \(C=1\).}
\label{fig:sigmaEps}
\end{figure}
Subsequent to forming the porosity parameter \(\varphi\) as the ratio of void and full domain, the injection of the magnitude approximation for \(a_{\epsilon}\) yields 
\begin{align}\label{eq:porosityCalc}
\varphi & = 
\frac{\left\vert \Omega_{\epsilon} \right\vert}{\left\vert \Omega\right\vert } 
= \frac{\left\vert \Omega - \bigcup\limits_{i=1}^{N\left(\epsilon\right)} Y_{S,{i}}^{\epsilon}\right\vert}{\left\vert \Omega\right\vert} 
= 1 - \frac{\left\vert \bigcup\limits_{i=1}^{N\left(\epsilon\right)} Y_{S,{i}}^{\epsilon}\right\vert}{\left\vert \Omega\right\vert} 
= 1 - \frac{\pi a_{\epsilon}^{3}}{6 \epsilon^{3}} = 
\begin{cases}
1 - \frac{C \pi}{6} \epsilon^{9}  \quad & \text{in case (i)},\\
1 - \frac{C \pi}{6} \epsilon^{6}  \quad & \text{in case (ii)},\\
1 - \frac{C \pi}{6} \epsilon^{3}  \quad & \text{in case (iii)}, \\
1 - \frac{C \pi}{6}  \quad & \text{in case (iv)} \\
\end{cases} \\
& \xrightarrow[\varepsilon \searrow 0]{} 
\begin{cases}
1  \quad & \text{in cases (i-iii)}, \\
1  - \frac{C \pi}{6} \quad & \text{in case (iv)}, \label{eq:porosity}
\end{cases} 
\end{align} 
where \(\left\vert \;\cdot\; \right\vert\) denotes the Lebesgue measure of the standard Euclidean space. 
The claim follows from setting \(C=1\) in \eqref{eq:porosity}. 
\end{proof}
\begin{remark}
Proposition~\ref{prop:porosity} frames the modeling possibilities of the presented approach, since the minimal attainable porosity is similar to a square sphere packing \cite{ghezzehei2003pore,nield2017convection}. 
The formal computations above thus imply physical reasoning for the theoretical homogenization limit equations in Conjecture~\ref{thmEvoNSE}.
\end{remark}
\begin{remark}
A lower porosity could be obtained e.g.\ by considering flow through two-dimensional porous media in three dimensions, or by choosing three-dimensional obstacles in different arrangements \cite{dalwadi2015understanding,nield2017convection}. 
Whereas the former becomes reasonable when modeling for example fibers as obstacles with a circular cross-section \cite{bang1999application} and repeating above calculations for \(d=2\), the latter renders rather complicated, due to the necessity of proving Conjecture \ref{thmEvoNSE} under loosened initial topological assumptions on the obstacles \cite{feireisl2016homogenization}. 
The question if all four cases would be retained under a differentiability-breaking change of shape or cell-crossing shifts in location, remains to be answered. 
\end{remark}

\begin{remark}\label{rem:unifiedBL}
Under Assumption~\ref{ass:thmAssump} we formulate a unified BL for case Conjecture~\ref{thmEvoNSE}(ii) below (Definition~\ref{def:homogenizedNSEtarget}), which depends on \(\sigma\) and formally limits 
\begin{itemize}
\item either to the nonstationary nonlinear NSE in case (i) for \( \sigma \to \infty\) 
\item or (via rescaling the solution to \(\tilde{\bm{u}}/\sigma_{\epsilon}^{2}\)) to the stabilized DL in case (iii) and (iv) for \(\sigma \searrow 0\).
\end{itemize}
Further, since the continuity in the low volume fraction limit \cite{allaire1991continuity} implies that \(\mathbf{M}^{-1}\) is the limit of \(\mathbf{A}\), we use \(\mathbf{A}\) in the modified BL and, due to Assumption~\ref{ass:thmAssump}(3.), reduce it to its single eigenvalue \(A\). 
In summary, the resulting model equation is assumed to be valid for all herein considered porosities and permeabilities. 
We additionally motivate the procedure of emulating all four regimes by recent observations that turbulence prevails for porosity values approaching unity in aligned arrays of spheres \cite{rao2022possibility}. 
\end{remark}
\begin{definition}\label{def:homogenizedNSEtarget}
Based on Conjecture~\ref{thmEvoNSE}, Assumption~\ref{ass:thmAssump}, Proposition~\ref{prop:porosity} and Remark~\ref{rem:unifiedBL}, we construct a unified nonstationary nonlinear BL
\begin{align}\label{eq:targetEBTL}
\begin{cases}
	\partial_{t} \bm{u} +  \bm{u} \cdot \bm{\nabla}_{\bm{x}} \bm{u} - \nu \bm{\Delta}_{\bm{x}} \bm{u} + \frac{\nu}{\sigma^{2}} A^{-1} \bm{u} = \bm{F} - \bm{\nabla}_{\bm{x}} p  \quad &\text{in } \Omega  \times I, \\
	\mathrm{div}_{\bm{x}} \bm{u} = 0 										   \quad &\text{in } \Omega \times I , \\
    \bm{u}\vert_{t=0} = \bm{u}_{0} \quad &\text{in } \Omega, \\
	\bm{u} = \bm{0} 													   \quad &\text{on } \partial\Omega  \times I, 
\end{cases} 
\end{align} 
which is used as a target PDE system for the kinetic model derivation below in Section~\ref{sec:kinMod}, and is to be approximated with LBMs in the sequel \cite{simonis2023hlbmPartII}. 
Due to the unified perspective, the PDE system~\eqref{eq:targetEBTL} is now referred to as homogenized NSE (HNSE). 
\end{definition}

\section{Kinetic Model Derivation}\label{sec:kinMod}

The overall goal of this series of works is to derive consistently an LBE for approximating the HNSE, which then forms the centerpiece of the final LBM algorithm. 
Conforming to the discretization approach of LBM, we couple one or more scaling parameters (e.g.\ \(\varepsilon>0\)) of a Boltzmann-like equation to an artificially injected grid parameter $h\in\Rzsp$. 
Toward this aim, in the present work we construct this kinetic equation in undiscretized form which approximates the HNSE in a diffusive limit. 
The aim of the sequel (part II \cite{simonis2023hlbmPartII}) is then, to discretize the kinetic model while retaining the kinetic limit (see Figure~\ref{fig:limitCons}). 
The discretization is thus required to be limit consistent in the sense of \cite{simonis2022limit,simonis2023pde}. 
Since we aim for formal convergence of the final LBM only, the notion of limit consistency requires formal convergence of the kinetic model as well which is proven below. 

\subsection{Preliminaries}
Let $\Omega \subseteq \mathbb{R}^d$ with \(d=3\) be a volume of rarefied gas which comprises many interacting particles. 
Via equalizing the mass $m \in \mathbb{R}_{>0}$, we interpret the particles as point masses. 
The state of a one-particle system is assumed to depend on position $\bm{x} \in \Omega$ and velocity $\bm{v} \in \Xi$ at time $t \in I=[t_0,t_1] \subseteq \mathbb{R}$ with \(T\geq t_{1}>t_{0}>0\), where $\Omega \subseteq \mathbb{R}^d$ denotes the positional space, $\Xi = \mathbb{R}^d$ is the velocity space, \(\mathfrak{P} \coloneqq \Omega \times \Xi\) is the phase space, and the crossing \(\mathfrak{R} \coloneqq \Omega \times \Xi \times I \) defines the phase-time tuple. 
\begin{definition}
The probability density function 
\begin{align} \label{eq:statisticPDF}
 f\colon\; \mathfrak{R} \to \mathbb{R}_{>0} ,\, (\bm{x},\bm{v},t) \mapsto f(\bm{x} , \bm{v}, t ) 
\end{align}
for the particles' positions \(\bm{x}\in \Omega\) and velocities \(\bm{v}\in \Xi\) at time \(t\in I\) defines the state of the dynamical system which is governed by the Boltzmann equation (BE)
\begin{align} \label{eq:continuousBE}
  \left( \partial_{t} + \bm{v} \cdot \bm{\nabla}_{\bm{x}} + \frac{\bm{F}}{m} \cdot \bm{\nabla}_{\bm{v}} \right) f = J(f,f ) \quad \text{in } \mathfrak{R}, 
\end{align}
where 
\begin{align}
f \vert_{t=0} = f_{0} \quad \text{in } \mathfrak{P}
\end{align}
supplements a suitable initial condition. 
The operator 
\begin{align}
J\left( f, f\right) =  \int_{\mathbb{R}^{3}} \int_{S^{2}} 
\vert \bm{v} - \bm{w} \vert 
\left[ 
f \left( \bm{x}, \bm{v}^{\prime}, t \right) 
f \left( \bm{x}, \bm{w}^{\prime}, t \right) 
- 
f \left( \bm{x}, \bm{v}, t \right) 
f \left( \bm{x}, \bm{w}, t \right) 
\right]
\,\mathrm{d}\bm{N} 
\,\mathrm{d}\bm{w}
\end{align}
models the collision, where \(\mathrm{d}\bm{N}\) is the normalized surface integral with the unit vector \(\bm{N}\in S^{2}\) and \(\left( \bm{v}^{\prime}, \bm{w}^{\prime} \right)^{\mathrm{T}} = T_{\bm{N}}\left( \bm{v}, \bm{w} \right)^{\mathrm{T}}\) result from the transformation \(T_{\bm{N}}\) that models hard sphere collision \cite{babovsky1998boltzmann}. 
\end{definition}
\begin{definition}\label{def:moments}
Let \(f\) be given in the sense of \eqref{eq:statisticPDF}. 
Then, via prefactored integration over \(\Xi = \mathbb{R}^{d}\), we define the moments 
\begin{align}
n_f &\colon\;\begin{cases}
\Omega \times I \to\mathbb{R}_{>0},  \\
(\bm{x},t)\mapsto n_f(\bm{x},t) \coloneqq \int_{\mathbb{R}^d} f(\bm{x},\bm{v},t) \,\mathrm{d}\bm{v}, 
\end{cases} \label{eq:statisticPDFparticleDensity} \\
\rho_f &\colon\;\begin{cases} 
\Omega \times I \to\mathbb{R}_{>0}, \\
(\bm{x},t)\mapsto \rho_f(\bm{x},t) \coloneqq mn_f(\bm{x},t), 
\end{cases} \label{eq:statisticPDFmassDensity} \\
\bm{u}_f &\colon\;\begin{cases} 
\Omega \times I \to\mathbb{R}^{d}, \\
(\bm{x},t)\mapsto u_f(\bm{x},t) \coloneqq \frac{1}{n_f(\bm{x},t)} \int_{\mathbb{R}^d} \bm{v} f(\bm{x},\bm{v},t) \,\mathrm{d}\bm{v}, 
\end{cases} \label{eq:statisticPDFvelocity} \\
\mathbf{P}_f &\colon\;\begin{cases} 
\Omega \times I \to\mathbb{R}^{d\times d},  \\
(\bm{x},t) \mapsto \mathbf{P}_f(\bm{x},t) \coloneqq m\int_{\mathbb{R}^d} \left[\bm{v} - \bm{u}_f(\bm{x},t)\right] \otimes \left[\bm{v} - \bm{u}_f(\bm{x},t)\right] f(\bm{x},\bm{v},t)\,\mathrm{d}\bm{v}, 
\end{cases} \label{eq:statisticPDFstress} \\
p_f &\colon\; \begin{cases} 
\Omega \times I  \to \mathbb{R}_{>0}, \\
(\bm{x},t) \mapsto p_f(\bm{x},t) \coloneqq \frac{1}{d} \sum\limits^d_{i=1} \left(\mathbf{P}_f\right)_{i,i}(\bm{x},t), 
\end{cases}
\label{eq:statisticPDFpressure} 
\end{align}
respectively as particle density, mass density, velocity, stress tensor, and pressure. 
Here and below, the moments of $f$ are indexed with \(\cdot_{f}\). 
\end{definition}
Notably, the absolute temperature $\theta$ is determined implicitly by an ideal gas assumption \begin{align}\label{eq:idealGasLaw}
p_f=n_f R\theta,
\end{align}
where $R>0$ is the universal gas constant. 
To a dedicated order of magnitude in characteristic scales, the above moments approximate the macroscopic quantities conserved by the incompressible NSE \cite{gorban2018hilbert}. 
Equilibrium states $f^{\mathrm{eq}}$, defined by  
\begin{align} 
  J(f^{\mathrm{eq}}, f^{\mathrm{eq}}) =0 \quad \text{in } \mathfrak{R} ,
\end{align} 
exist \cite{gorban2018hilbert}. 
Via the gas constant \(R = k_{\mathrm{B}}/m \in\mathbb{R}_{>0}\) (where \(k_{\mathrm{B}}\in\mathbb{R}_{>0}\) is the Boltzmann constant) and $\theta\in\mathbb{R}_{>0}$, $n_f$ as well as $\bm{u}_f$, the equilibrium state is found to be of Maxwellian form
\begin{align} \label{eq:feq}
	f^{\mathrm{eq}}(\bm{x},\bm{v},t) \colon\; 
	\begin{cases}
	\mathfrak{R} \to \mathbb{R}, \\
	(\bm{x},\bm{v},t) \mapsto \frac{n_f(\bm{x},t)}{\left(2 \pi R \theta \right)^{\frac{d}{2}}} \exp\left(-\frac{\left[ \bm{v} - \bm{u}_f(\bm{x},t) \right]^2}{2 R \theta}\right) .
	\end{cases}
\end{align}
\begin{remark}
We identify $f^{\mathrm{eq}}/n_f$ as $d$-dimensional normal distribution for $\bm{v}\in\mathbb{R}^d$ with expectation $\bm{u}_f$ and covariance $R\theta \mathbf{I}_d$. 
In this regard, the arguments of \(f^{\mathrm{eq}}\) regularly appear in terms of moments \(f^{\mathrm{eq}} ( n_{f}, \bm{u}_{f}, \theta)\) (see e.g.\ \cite{krause2010fluid,he1997theory,junk2005asymptotic,lallemand2000theory}). 
\end{remark}
\begin{lemma}
The moments \(\rho_{f}\), \(\bm{u}_{f}\) and \(p_{f}\) are conserved by collision. 
\end{lemma}
\begin{proof}
From $f^{\mathrm{eq}}/n_f$ being a density function, we find
\begin{align} \label{eq:boltzStatisticEquilibrium7}
  \rho_{f^{\mathrm{eq}}} &\stackrel{\eqref{eq:statisticPDFmassDensity}}{=} m \int_{\mathbb{R}^d} f^{\mathrm{eq}} (\bm{x},\bm{v},t)\,\mathrm{d}\bm{v}~= mn_f =\rho_f , \\
\label{eq:boltzStatisticEquilibrium8}
  \bm{u}_{f^{\mathrm{eq}}} &\stackrel{\eqref{eq:statisticPDFvelocity}}{=} \frac{1}{n_{f^{\mathrm{eq}}}} \int_{\mathbb{R}^d} \bm{v} f^{\mathrm{eq}} (\bm{x},\bm{v},t)\,\mathrm{d}\bm{v}~= \bm{u}_f .
\end{align}
The covariance matrix of $f^{\mathrm{eq}}/n_f$ for a perfect gas \eqref{eq:idealGasLaw}, verifies the conservation of pressure 
\begin{align} \label{eq:boltzStatisticEquilibrium10}
  p_{f^{\mathrm{eq}}} & \stackrel{\eqref{eq:statisticPDFpressure}}{=}  \frac{1}{d}m \int_{\mathbb{R}^d} \left(\bm{v} - \bm{u}_{f^{\mathrm{eq}}}\right)^{2} f^{\mathrm{eq}} (\bm{x},\bm{v},t)\,\mathrm{d} \bm{v}  \nonumber \\
 & =  \frac{1}{d}m \int_{\mathbb{R}^d} \left(\bm{v} - \bm{u}_{f}\right)^{2} f^{\mathrm{eq}} (\bm{x},\bm{v},t)\,\mathrm{d}\bm{v}  \nonumber \\
 & =  \frac{1}{d} m n_f \sum_{i=1}^{d} R\theta \nonumber\\
 & = p_f. 
\end{align}
\end{proof}
\begin{definition}
According to the Bhatnagar--Gross--Krook (BGK) model \cite{bhatnagar1954model}, we simplify the collision operator $J$ in \eqref{eq:continuousBE} to  
\begin{align} \label{eq:collisionOperatorQ}%\label{eq: boltz statistic bgk 1}
  Q(f) \coloneqq -\frac{1}{\tau}(f-M_{f}^{\mathrm{eq}}) \quad & \text{in } \mathfrak{R} ,  
\end{align}
where $\tau > 0$ denotes the relaxation time between collisions, and \(M^{\mathrm{eq}}_{f} = f^{\mathrm{eq}}(\bm{x},\bm{v},t)\) is a formal particular Maxwellian determined by \(n_{f}\) and \(\bm{u}_{f}\).
\end{definition}
\begin{remark}
The conservation of both, \(\rho_{f}\) and \(\bm{u}_{f}\), respectively \eqref{eq:boltzStatisticEquilibrium7} and \eqref{eq:boltzStatisticEquilibrium8}, is upheld, since \(\ln(M^{\mathrm{eq}}_{f})\) is a collision invariant of \(Q\) (cf.\ Theorem~1.5 in \cite{krause2010fluid}).
\end{remark}
\begin{definition}
With $Q$ from \eqref{eq:collisionOperatorQ} implanted in \eqref{eq:continuousBE}, the BGK Boltzmann equation (BGKBE) reads
\begin{align}\label{eq:BGKBE}
  \underbrace{\left( \partial_t + \bm{v} \cdot \bm{\nabla}_{\bm{x}} + \frac{\bm{F}}{m} \cdot \bm{\nabla}_{\bm{v}}\right) }_{=~\frac{\mathrm{D}}{\mathrm{D}t}}  f = Q(f) \quad & \text{in } \mathfrak{R} , 
\end{align}
where \(\mathrm{D}/(\mathrm{D} t)\) is referred to as material derivative, and \( f (\cdot , \cdot , 0  ) = f_{0} \) sets a suitable initial condition. 
Here and below, the variable \(f\) is renamed to obey \eqref{eq:BGKBE} instead of \eqref{eq:continuousBE}. 
\end{definition}
\begin{remark}
Mostly under strict assumptions, several existence and uniqueness results for solutions to \eqref{eq:BGKBE} have been proven in the past. 
The global existence of solutions to the BGKBE \eqref{eq:BGKBE} has been rigorously proven in \cite{perthame1989global}. 
Weighted \(L^{\infty}\) bounds and uniqueness have later been established on bounded domains \cite{perthame1993weighted} and in \(\mathbb{R}^{d}\) \cite{mischler1996uniqueness}. 
Moreover, hydrodynamic scaling limits of the moments \(\rho_{f}\) and \(\bm{u}_{f}\) toward Leray's weak solutions of the incompressible NSE \cite{leray1934mouvement} have been rigorously proven by Saint-Raymond~\cite{saint-raymond2003bgk}. 
Below, we will refer to this type of scaling limit as diffusive instead, due to the presence of diffusion terms in the macroscopic limit. 
\end{remark}

\subsection{Homogenized BGK Boltzmann collision} 
As common to classical derivations in LBMs, we start with the mesoscopic viewpoint to formally assess the continuum limit toward the macroscopic TEQ. 
Apart from the procedure itself being classical, to the knowledge of the authors, the results below are novel. 
Let \(K\) denote the single eigenvalue of the permeability tensor \(\mathbf{A}\) according to the Definition~\ref{def:homogenizedNSEtarget}. 
All other definitions follow the notation in \cite{simonis2022limit,simonis2023pde}. 
\begin{definition}
Based on \eqref{eq:feq}, the homogenized Maxwellian for the BGK collision \eqref{eq:collisionOperatorQ} is defined as 
\begin{align}\label{eq:porousVeloMaxwellian}
M_f^{\mathrm{eq}} = f^{\mathrm{eq}}(n_f,\varpi \bm{u}_f,T) 
\end{align}
with an additional prefactor called porosity control 
\begin{align} \label{eq: d definition}
 \varpi= 1 - \nu\tau K^{-1} 
\end{align}
in the velocity argument. 
For $\varpi=1$, the collision reduces to the classical BGK operator \eqref{eq:collisionOperatorQ}.
\end{definition}
For any $\rho_f$, $\bm{u}_f$, $T$ and $\varpi$ we obtain the zeroth, first and second order balance laws
\begin{align} 
   \rho_{M^{\mathrm{eq}}_f} = m \int_{\mathbb{R}^d} M^{\mathrm{eq}}_f \,\mathrm{d}\bm{v} & = m \int_{\mathbb{R}^d} f^{\mathrm{eq}} (n_f, \varpi \bm{u}_f, T)\,\mathrm{d}\bm{v} =\rho_f, \label{eq: hbgk statistic equilibrium 1}\\
\bm{u}_{M^{\mathrm{eq}}_f} = \frac{1}{n_{M^{\mathrm{eq}}_f}} \int_{\mathbb{R}^d} \bm{v} M^{\mathrm{eq}}_f \,\mathrm{d}\bm{v} &= \frac{1}{n_f} \int_{\mathbb{R}^d} \bm{v} f^{\mathrm{eq}} (n_f, \varpi \bm{u}_f, T)\,\mathrm{d}\bm{v} = \varpi \bm{u}_f , \label{eq: hbgk statistic equilibrium 2}\\
 p_{M^{\mathrm{eq}}_f} = \frac{m}{d} \int_{\mathbb{R}^d} \left(\bm{v} - \bm{u}_{M^{\mathrm{eq}}_f}\right)^{2} M^{\mathrm{eq}}_{f} \,\mathrm{d} \bm{v}  & = \frac{m}{d} \int_{\mathbb{R}^d} \left(\bm{v} - \varpi \bm{u}_f\right)^{2} f^{\mathrm{eq}} (n_f, \varpi \bm{u}_f, T) \,\mathrm{d}\bm{v} = p_f, \label{eq: hbgk statistic equilibrium 3}
\end{align}
respectively. 
Notably, the hydrodynamic first order moment of \(M^{\mathrm{eq}}_{f}\) in \eqref{eq: hbgk statistic equilibrium 2} differs from the one of \(f^{\mathrm{eq}}\) due to the prefactored porosity control \(\varpi\).
\begin{definition}
With the homogenized $M_{f}^{\mathrm{eq}}$ from \eqref{eq:porousVeloMaxwellian} implanted in \eqref{eq:collisionOperatorQ}, the homogenized BGKBE (HBGKBE) reads
\begin{align}\label{eq:HBGKBE}
  \frac{\mathrm{D}}{\mathrm{D}t} f = Q(f) \quad & \text{in } \mathfrak{R} .
\end{align}
\end{definition}

\begin{remark}
In the present mesoscopic framework, the term homogenized refers to generalizing the BGKBE as a special case for \(\varpi = 1\) (via \(K\to \infty\)) to a broader validity where \(\varpi \neq 1\). 
Below, we formally indicate that for \(\varpi < 1\) the homogenized Maxwellian \eqref{eq:porousVeloMaxwellian} leads to imposing a nonstandard hydrodynamic similarity of the HBGKBE to the HNSE in the broadest sense of Hilbert's sixth problem. 
The artificial case of \(\varpi > 1\) is neglected hereafter. 
\end{remark}

\subsection{Homogenized diffusive limit}
Analogously to the derivation in \cite{simonis2022limit}, we relate the HBGKBE \eqref{eq:HBGKBE} to the HNSE \eqref{eq:targetEBTL} in the sense of diffusive limiting. 
To this end, we formally verify that the assumed to be well-defined moments in Definition \ref{def:moments} obey the balance equations of the HNSE. 
The derivation is done in three steps (see \cite{simonis2022limit}). 

\subsubsection{Step 1: Mass conservation and momentum balance} 
Let $f^\star$ be a solution to the HBGKBE \eqref{eq:HBGKBE}.
Multiplying \eqref{eq:HBGKBE} by $m$ and integrating over the velocity space $\Xi=\Rz^d$ yields the solenoidal constraint in \eqref{eq:targetEBTL} after division by the constant \(\rho_{f^{\star}}\), where the force term vanishes when applying Corollary~5.2 from \cite{krause2010fluid} with $g = 1$ and $\bm{a}=\bm{F}$ in the respective notation. 
To balance momentum, we integrate \(m \bm{v} \times \)\eqref{eq:HBGKBE} over the $\Xi=\Rz^d$ and obtain in $\Omega_{T}$ that
\begin{align} \label{eq:momentumHLBM}
\partial_ t \left( \rho_{f^\star} \bm{u}_{f^\star} \right) + \bm{\nabla}_{\bm{x}} \cdot \mathbf{P}_{f^\star} + \left(\rho_{f^\star}\bm{u}_{f^\star} \cdot \bm{\nabla}_{\bm{x}}\right) \bm{u}_{f^\star} + \bm{F} &= -\frac{1}{\tau} \left( \rho_{f^\star} \bm{u}_{f^\star} - \varpi  \rho_{f^\star} \bm{u}_{f^\star} \right) \nonumber \\
 &= - \nu K^{-1} \rho_{f^\star} \bm{u}_{f^\star}.
\end{align}
Besides the homogenization term on the right hand side, the derivation of \eqref{eq:momentumHLBM} closely follows the procedure in \cite{simonis2022limit,simonis2023pde}. 
In the end, via \eqref{eq:momentumHLBM}\(/ \rho_{f^\star}\) we recover a balance law of momentum in conservative form where the additional term $-\nu K^{-1} \bm{u}_{f^\star}$ is induced by the homogenization controlled equilibrium and corresponds to \( - (\nu / (\sigma^2) )\mathbf{A}^{-1} \bm{u} \) under the assumptions on the porous structure made above. 
Hence, with a suitably defined \(\mathbf{P}_{f^\star}\) conforming to the assumptions of incompressible Newtonian flow, the HNSE~\eqref{eq:targetEBTL} is reached in the diffusive limit. 
This incompressible limit regime of the HBGKBE \eqref{eq:HBGKBE} arises from parameter alignment to diffusion terms. 
We thus extend the derivation given in \cite{krause2010fluid,simonis2022limit} for the classical BGKBE to the HBGKBE. 

\subsubsection{Step 2: Incompressible limit}
We recall the definitions and assignments made in \cite{simonis2022limit}, i.e.\ the incompressible limit regime of the BGKBE \eqref{eq:BGKBE} is obtained via aligning parameters to the diffusion terms \cite{krause2010fluid,simonis2022limit}. 
Based on that, we perform the same assignments here to obtain the HBGKBE~\eqref{eq:HBGKBE} in the diffusion limit. 
Let $l_\mathrm{f}$ be the mean free path, $\overline{c}$ the mean absolute thermal velocity, and \(\nu>0\) a kinematic viscosity. 
Assuming that a characteristic length $L$ and a characteristic velocity $U$ are given, we define the Knudsen number, the Mach number and the Reynolds number, respectively
\begin{align} 
\Kn & \coloneqq  \frac{ l_{\mathrm{f}} }{L} , \\
\Ma & \coloneqq \frac{U}{\bar{c}_\mathrm{s}} , \\
\Rey & \coloneqq \frac{U L}{\nu}.
\end{align}
These nondimensional numbers relate as 
\begin{align}
\Rey = \frac{l_{\mathrm{f}} \bar{c}_\mathrm{s}}{\nu} \frac{\Ma}{\Kn} =\sqrt{\frac{24}{\pi}} \frac{\Ma}{\Kn} ,
\end{align}
via defining \(\nu \coloneqq \pi \overline{c} l_{\mathrm{f}} / 8\) and the isothermal speed of sound \(\bar{c}_\mathrm{s} \coloneqq \sqrt{3 R\theta}\) (see also \cite{saint-raymond2003bgk} and references therein). 
\begin{definition}
To link the mesoscopic distributions with the macroscopic continuum we inversely substitute \(\bar{c}_{\mathrm{s}}\) with an artificial parameter \(\varepsilon \in \Rzsp\) through 
\begin{align}
\bar{c}_{\mathrm{s}} \mapsfrom \frac{1}{\varepsilon} . 
\end{align}
Here and below, the symbol \(\mapsfrom\) denotes the assignment operator. 
\end{definition}
In the limit $\varepsilon \searrow 0$, the incompressible continuum is reached, since $\Kn$ and $\Ma$ tend to zero while $\Rey$ remains constant \cite{saint-raymond2003bgk}. 
Based on that, we assign
\begin{align} \label{eq:cOverline}
\overline{c}= \sqrt{\frac{8k_{\mathrm{B}} \theta}{m \pi}} & \mapsfrom \sqrt{\frac{8}{3 \pi}} \frac{1}{\varepsilon} , \\
l_{\mathrm{f}} & \mapsfrom \sqrt{\frac{24}{\pi}} \nu \varepsilon
\end{align}
and \eqref{eq:cOverline} unfold the relaxation time  
\begin{align} %\label{eq:tau}
  \tau = \frac{l_{\mathrm{f}}}{\overline{c}} \mapsfrom 3 \nu \varepsilon^2 .
\end{align}
\begin{definition}
We reassign the so called porosity controller
\begin{align} \label{eq:bgkPorousity}
 \varpi \mapsfrom 1 - 3 \nu^2 \varepsilon^2  K^{-1}  \eqqcolon \varpi_{\varepsilon} ,
\end{align}
where \(\varepsilon>0\) is a scaling parameter, and define the \(\varepsilon\)-parametrized HBGKBE similarly to the \(\varepsilon\)-parametrized BGKBE in \cite{simonis2022limit} as
\begin{align} \label{eq: lbm discrete velocity 3}
\frac{\mathrm{D}}{\mathrm{D} t} f = - \frac 1 {3\nu \varepsilon^2} \left( f-M_{f}^{\mathrm{eq}} \right)  \quad \text{in } \mathfrak{R} , 
\end{align}
where the homogenized Maxwellian distribution evaluated at $(n_f,\varpi_\varepsilon\bm{u}_f)$ now reads 
\begin{align} \label{eq:homMaxwellianParametrized}
 M_{f}^{\mathrm{eq}} = \frac{n_f \varepsilon^d}{\left( \frac 2 3 \pi \right)^{\frac{d}{2}}} \exp\left( - \frac{3}{2} \left( \bm{v} \varepsilon - \varpi_\varepsilon \bm{u}_f  \varepsilon\right)^2\right) \quad \text{in } \mathfrak{R} .
\end{align}
\end{definition}
The HBGKBE \eqref{eq: lbm discrete velocity 3} is accordingly transformed to
\begin{align} \label{eq: boltz statistic limit 8}
 f & = M_{f}^{\mathrm{eq}} - 3\nu \varepsilon^2 \matD f \quad \text{in } \mathfrak{R} .
\end{align}
Repeating the material derivative through \((\matDil)\)\eqref{eq: boltz statistic limit 8} yields
\begin{align} \label{eq: boltz statistic limit 9}
 \matD f = \matD M_{f}^{\mathrm{eq}} - 3 \nu \varepsilon^2 \left( \matD \right)^{2} f  \quad \text{in } \mathfrak{R}.
\end{align}
The expression \eqref{eq: boltz statistic limit 9} serves to substitute $(\matDil) f$ in \eqref{eq: boltz statistic limit 8} which gives
\begin{align} %\label{eq: boltz statistic limit 10}
 f  =  M_{f}^{\mathrm{eq}} - 3 \nu \varepsilon^2 \matD M_{f}^{\mathrm{eq}} + \left( 3\nu \varepsilon^2 \matD \right)^{2} f   \quad \text{in } \mathfrak{R}.
\end{align}
Repeating the above subsequently produces higher order terms and substitutions. 
The evolving family unfolds the power series 
\begin{align} \label{eq: boltz statistic limit 11}
 f  = \sum\limits_{i=0}^\infty  \left( -3\nu \varepsilon^2 \matD \right)^i M_{f}^{\mathrm{eq}} \quad \text{in } \mathfrak{R} .
\end{align}

\subsubsection{Step 3: Newton's hypothesis}
To complete the macroscopic limit the stress tensor $\mathbf{P}_{f^\star}$ in \eqref{eq:momentumHLBM} has to be matched to \eqref{eq:targetEBTL}, which for a solution \(f^{\star}\) to the HBGKBE~\eqref{eq:HBGKBE} yields
\begin{align}\label{eq:newtonsHypHLBM}
  \mathbf{P}_{f^\star} = - p_{f^\star} \mathbf{I} + 2 \nu \rho \mathbf{D}_{f^{\star}} + \mathcal{O}\left( \varepsilon^{b}\right) \quad \text{in } \Omega_{T}
\end{align}
up to an order \(b>0\). 
Using \eqref{eq: boltz statistic limit 11}, an approximation ansatz of the form
\begin{align} \label{eq: boltz statistic limit 12}
 f ^\star = M_{f^\star}^{\mathrm{eq}} - 3\nu \varepsilon^2 \matD M_{f^\star}^{\mathrm{eq}} \quad \text{in } \mathfrak{R}
\end{align}
is chosen. 
As before, this choice is based upon the assumption that higher order terms are sufficiently small for $\varepsilon \to 0$ such that the order \(b\) in turn is large enough. 
To verify \eqref{eq:newtonsHypHLBM}, we compute the stress tensor according to its definition \eqref{eq:statisticPDFstress}. 
In the following, \(f\)-indices at physical moment expressions are omitted for the sake of simplicity. 
At first, we substitute the material derivative and use the mass conservation to obtain 
\begin{align}\label{eq: boltz statistic limit 13}
\matD M_{f}^{\mathrm{eq}} 
&= 
\left( 
 			\frac{1}{\rho} \matD \rho 
 			+ 3 \varepsilon^2 \varpi_\varepsilon \bm{c}_{\varpi} \cdot \matD \bm{u} 
			- \frac{3 \varepsilon^2\bm{c}_{\varpi}}{m} \cdot \bm{F}  			
 			\right) M_{f}^{\mathrm{eq}}  \nonumber \\
&= 
\biggl[
			\frac{1}{\rho} \left( \partial_{t} + \bm{v}  \cdot  \bm{\nabla}_{\bm{x}} \right) \rho \nonumber\\
   &\hphantom{= \biggl[} 
			+ 3 \varepsilon^2 \varpi_\varepsilon \bm{c}_{\varpi}  \cdot   \left( \partial_{t} + \bm{v}  \cdot  \bm{\nabla}_{\bm{x}} \right) \bm{u} 
			- \frac{3 \varepsilon^2 \bm{c}_{\varpi}}{m}  \cdot  \bm{F}		
			\biggr] M_{f}^{\mathrm{eq}} \nonumber \\
&= 
\biggl[
			\frac{1}{\rho} \left( -\bm{u} \cdot \bm{\nabla}_{\bm{x}}\rho - \rho \bm{\nabla}_{\bm{x}} \cdot \bm{u}  + \bm{v}  \cdot  \bm{\nabla}_{\bm{x}} \rho \right) \nonumber\\
   &\hphantom{= \biggl[} +  3 \varepsilon^2 \varpi_\varepsilon \bm{c}_{\varpi}  \cdot   \left( \partial_{t} + \bm{v}  \cdot  \bm{\nabla}_{\bm{x}} \right) \bm{u} 
			- \frac{3 \varepsilon^2 \bm{c}_{\varpi}}{m}  \cdot  \bm{F} 			
  			\biggr] M_{f}^{\mathrm{eq}} \nonumber \\
&= \biggl[ 
			-\underbrace{\bm{\nabla}_{\bm{x}} \cdot \bm{u}}_{=:~a_f} 
			+ \smash{\underbrace{\frac{\bm{c}}{\rho}  \cdot  \bm{\nabla}_{\bm{x}} \rho}_{=:~b_f} 
			+ \underbrace{3 \varepsilon^2\varpi_\varepsilon \bm{c}_{\varpi} \cdot \partial_{t}\bm{u}}_{=:~c_f}} \nonumber\\
   &\hphantom{= \biggl[}
			+ \smash{\underbrace{3 \varepsilon^2\varpi_\varepsilon \bm{c}_{\varpi} \cdot  \left(\bm{v}  \cdot  \bm{\nabla}_{\bm{x}}\right) \bm{u}}_{=:~d_f}} 
			- \underbrace{\frac{3 \varepsilon^2\bm{c}_{\varpi}}{m}  \cdot  \bm{F} }_{=:~e_{f}}		
			\biggr] M_{f}^{\mathrm{eq}} 
\end{align}
in $\mathfrak{R}$, where 
\begin{align}\label{eq:relVeloHom}
\bm{c} & \coloneqq \bm{v}-\bm{u}, \\
\bm{c}_{\varpi} & \coloneqq  \bm{v}-\varpi_{\varepsilon} \bm{u}, 
\end{align}
are relative velocities, i.e.\ the deviation of the particle velocities \(\bm{v}\) from the local mean \(\bm{u}\). 
Inserting the derivative \eqref{eq: boltz statistic limit 13} in \eqref{eq: boltz statistic limit 12} yields
\begin{align} \label{eq: boltz statistic limit 14}
 f &= M_{f}^{\mathrm{eq}}\left[1 - 3 \varepsilon^{2} \nu  \left(-a_f+b_f+c_f+d_f+e_f\right)\right] \quad \text{in } \mathfrak{R}.
\end{align}
Secondly, we evaluate the velocity space integrals of the individual terms \(a_{f}, b_{f}, \ldots , e_{f}\). 
To this end, we use the symmetric properties of $M_{f}^{\mathrm{eq}}$ and the fact that $M_{f}^{\mathrm{eq}}/n$ is a normal distribution with covariance matrix $1/(3 \varepsilon^2) \mathbf{I}_{d}$. 
In $\Omega_{T}$ and for any $i,j,k,l\in \left\{1,2,...,d\right\}$ we verify that
\begin{align} %\label{eq: boltz statistic limit 13}
 m \int_{\Rz^d} c_i c_j M_{f}^{\mathrm{eq}}\,\mathrm{d}\bm{v} 
 &= 
  m\int_{\Rz^d} c_{\varpi,i} c_{\varpi,j} M_{f}^{\mathrm{eq}} \,\mathrm{d}\bm{v} \nonumber \\
  & \quad\quad - m\int_{\Rz^d} \left( 1-\varpi_\varepsilon\right) u_i \left[ 2 v_j   - \left(1 +  \varpi_\varepsilon \right) u_j \right] M_{f}^{\mathrm{eq}}\,\mathrm{d}\bm{v}  \nonumber \\
 &= \frac{\rho}{3 \varepsilon^2} \delta_{ij} -  \rho \left( 1-\varpi_\varepsilon \right)^{2}  u_i u_j \nonumber \\
 &= p \delta_{ij} + \mathcal{O} ( \varepsilon^{4} ), \label{eq:ccH} 
\end{align}
as well as
\begin{align}
m \int_{\Rz^d} c_i c_j c_k M_{f}^{\mathrm{eq}}\,\mathrm{d}\bm{v} 
&=  \underbrace{m\int_{\Rz^d} c_{\varpi,i} c_{\varpi,j} c_{\varpi,k} M_{f}^{\mathrm{eq}} \,\mathrm{d}\bm{v}}_{=~0}  \nonumber \\
										    & \quad\quad + \left( \varpi_\varepsilon - 1\right)\rho u_i \left[ \frac{1}{3 \varepsilon^2} \delta_{jk} -  \left( 1-\varpi_\varepsilon \right)^{2}  u_j u_k \right] \nonumber \\
										    & \quad\quad + \left( \varpi_\varepsilon - 1\right)\rho u_j \left[ \frac{1}{3 \varepsilon^2} \delta_{ik} -  \left( 1-\varpi_\varepsilon \right)^{2}  u_i u_k \right] \nonumber \\
										    & \quad\quad + \left( \varpi_\varepsilon - 1\right)\rho u_k \left[ \frac{1}{3 \varepsilon^2} \delta_{ij} -  \left( 1-\varpi_\varepsilon \right)^{2}  u_i u_j \right] \nonumber \\
										    &= \sum\limits_{\stackrel{\alpha \beta \gamma~\in}{\left\{ ijk, jik, kij\right\}} } \left\{ \left( \varpi_\varepsilon - 1\right) \rho u_\alpha \left[ \frac{1}{3 \varepsilon^2} \delta_{\beta\gamma} -  \left( 1-\varpi_\varepsilon \right)^{2}  u_\beta u_\gamma \right] \right\} \nonumber \\
												&= \mathcal{O}(1) , \label{eq:cccH} 
\end{align}
and
\begin{align}m\int_{\Rz^d} c_i c_j c_{\varpi,k} v_l M_{f}^{\mathrm{eq}}\,\mathrm{d}\bm{v} 
                        &= \rho                                      	 	\biggl\{ \frac{1}{9 \varepsilon^4} \left( \delta_{ij}\delta_{kl} + \delta_{ik}\delta_{jl} + \delta_{il}\delta_{jk} \right) \nonumber \\    
                                              	 &  \quad\quad+ \left( \varpi_\varepsilon - 1 \right)^{2} u_i u_j \left[ \frac{1}{3 \varepsilon^2} \delta_{kl} - \left(1-\varpi_\varepsilon\right)^{2} u_k u_l \right]  \nonumber \\
                                              	 &  \quad\quad+ \left( \varpi_\varepsilon - 1 \right) \varpi_\varepsilon u_i u_l \left[ \frac{1}{3 \varepsilon^2} \delta_{jk} - \left(1-\varpi_\varepsilon\right)^{2} u_j u_k \right]  \nonumber \\ 
                                              	 &  \quad\quad+ \left( \varpi_\varepsilon - 1 \right) \varpi_\varepsilon u_j u_l \left[ \frac{1}{3 \varepsilon^2} \delta_{ik} - \left(1-\varpi_\varepsilon\right)^{2} u_i u_k \right]  \biggr\}, \nonumber \\ 																			&=  \frac{\rho}{9 \varepsilon^4} \left( \delta_{ij}\delta_{kl} + \delta_{ik}\delta_{jl} + \delta_{il}\delta_{jk} \right) + \mathcal{O}(1) . \label{eq:cccvH}				
\end{align}
The order estimates hold since, by construction $\varpi_\varepsilon-1 \in \mathcal{O}(\varepsilon^2)$. 
Hence, we obtain
\begin{align}										     
 m \int_{\Rz^d} c_i c_j a_f M_{f}^{\mathrm{eq}}\,\mathrm{d}\bm{v} &= \left( m \int_{\Rz^d} c_i c_j M_{f}^{\mathrm{eq}}\,\mathrm{d}\bm{v} \right) \partial_{x_k} u_k \nonumber \\
 										         &\stackrel{\eqref{eq:ccH}}{=} \left(\frac{\rho}{3 \varepsilon^2} + \mathcal{O}(\varepsilon^{4})\right) \partial_{x_k} u_k\nonumber \\
														 &= \frac{\rho}{3 \varepsilon^2} \partial_{x_k} u_k + \mathcal{O}(\varepsilon^{4}) , \\
 m \int_{\Rz^d} c_i c_j b_f M_{f}^{\mathrm{eq}}\,\mathrm{d}\bm{v} &= \left( m \int_{\Rz^d} c_i c_j c_k M_{f}^{\mathrm{eq}}\,\mathrm{d}\bm{v} \right) \frac{1}{\rho}\partial_{x_k} \rho \nonumber \\											 &\stackrel{\eqref{eq:cccH}}{=} \mathcal{O}(1) , \\
 m \int_{\Rz^d} c_i c_j c_f M_{f}^{\mathrm{eq}}\,\mathrm{d}\bm{v} &= \left( m \int_{\Rz^d} c_i c_j c_k M_{f}^{\mathrm{eq}}\,\mathrm{d}\bm{v} \right) 3\varepsilon^2\varpi_\varepsilon \partial_{t} u_k \nonumber \\ 
 &\stackrel{\eqref{eq:cccH}}{=} \mathcal{O}(\varepsilon^{2}) , \\
 m\int_{\Rz^d} c_i c_j d_f M_{f}^{\mathrm{eq}}\,\mathrm{d}\bm{v} &= \left( m \int_{\Rz^d} c_i c_j c_{\varpi,k} v_l M_{f}^{\mathrm{eq}}\,\mathrm{d}\bm{v} \right) 3 \varepsilon^2 \varpi_\varepsilon \partial_{x_l} u_k  \nonumber \\
&\stackrel{\eqref{eq:cccvH}}{=} 3 \varepsilon^2 \varpi_\varepsilon \partial_{x_l} u_k \left[ 
\frac{\rho}{9 \varepsilon^4} \left( \delta_{ij}\delta_{kl} + \delta_{ik}\delta_{jl} + \delta_{il}\delta_{jk} \right) + \mathcal{O}(1) \right] \nonumber \\
&\stackrel{\eqref{eq:bgkPorousity}}{=} \partial_{x_l} u_k 
												\frac{\rho}{3 \varepsilon^2} \left( \delta_{ij}\delta_{kl} + \delta_{ik}\delta_{jl} + \delta_{il}\delta_{jk} \right) + \mathcal{O}(1), \\		
m\int_{\Rz^d} c_i c_j e_f M_{f}^{\mathrm{eq}}\,\mathrm{d}\bm{v} &= \rho \left(\varpi_\varepsilon -1\right) F_{k} u_{j}  \delta_{ik}\nonumber \\
&= O(\varepsilon^2) .
\end{align}
Third and finally, each \(\mathbf{P}\)-component $P_{ij}$ for $i,j \in \left\{1, 2, ..., d \right\}$ is computable in \(\mathfrak{R}\). 
Via reordering terms, we obtain 
\begin{align} %\label{eq: boltz statistic limit 15}
P_{ij} 
&=  m \int_{\Rz^d} c_i c_j  \left[ 1- 3 \nu \varepsilon^{2} \left(-a_f+b_f+c_f+d_f + e_{f} \right) \right] M_{f}^{\mathrm{eq}}\,\mathrm{d}\bm{v}   \nonumber\\
&= p \delta_{ij} 
 			-  3 \nu \varepsilon^{2} \left[
			- \frac{\rho}{3 \varepsilon^2} \partial_{x_k} u_k 
			+ \partial_{x_l} u_k \frac{\rho}{3 \varepsilon^2} \left( \delta_{ij}\delta_{kl} + \delta_{ik}\delta_{jl} + \delta_{il}\delta_{jk} \right) + \mathcal{O}(1)
		 \right]      \nonumber \\
        &= p \delta_{ij} + \nu \rho \left[ \delta_{ij} \partial_{x_k} u_k - \partial_{x_l} u_k \left( \delta_{ij}\delta_{kl} + \delta_{ik}\delta_{jl} + \delta_{il}\delta_{jk}\right) \right] + \mathcal{O}\left( \varepsilon^{2}\right) \nonumber \\
        &= p \delta_{ij} -  \nu \rho \left( \partial_{x_{i}} u_j + \partial_{x_j} u_i \right) + \mathcal{O}\left(\varepsilon^{2} \right) 
\end{align}
and thus equivalently
\begin{align}
\mathbf{P}  = p \mathbf{I}_{d} - 2 \nu \rho \mathbf{D}  + \mathcal{O}\left(\varepsilon^{2}\right) \quad \text{in } \Omega_{T}, 
\end{align}
which formally proves the approximate recovery of the HNSE~\eqref{eq:targetEBTL} in the hydrodynamic limit.

\section{Conclusion}\label{sec:conclusion}

The overall aim of this series of works is to construct HLBMs that approximate the governing equations for homogenized nonstationary nonlinear fluid flow through porous media. 

Summarizing the present work (part I), we make two contributions. 
At first, we recall the existing framework of Allaire for homogenizing the NSE with specific geometric configurations. 
We gather proven results towards a unified homogenization of incompressible nonstationary NSE in the framework of porous media as abstracted periodically arranged obstacles. 
We restate the stationary simplification and subsequently include time-dependency. 
In the latter case, we form a conjecture of four cases as the result of homogenization depending on the size of the obstacles: (i) nonstationary NSE, (ii) nonstationary BL, (iii) time-dependent DL, and (iv) DL with memory. 
We isolate the missing proofs and review existing results. 
Further, an application-oriented rationale is presented which determines the porosity range recoverable by the mathematical model. 
Based on that, we formulate a modified nonstationary BL which is termed HNSE and serves as a unified targeted PDE system for the HLBM to be proposed in the sequel (part II \cite{simonis2023hlbmPartII}). 
Second, as a first step toward the HLBM, we propose a kinetic model, the HBGKBE, which approximates the nonstationary HNSE in a diffusive scaling limit. 
We formally prove that the zeroth and first order moments of the kinetic model provide solutions to the mass and momentum balance variables of the macroscopic model up to certain orders in the scaling parameter.  
Specifically, the stress tensor is approximated with \(O(\varepsilon^{2})\) in the diffusive limit. 

Future studies with respect to mathematical and kinetic model extensions should involve mixed boundary conditions at the porous matrix \cite{fabricius2017homogenization}, porous--void interface conditions \cite{griebel2009homogenisation}, porosity gradients in the solid matrix \cite{dalwadi2015understanding} or investigating modeling possibilities of anisotropic permeability tensors \cite{bang1999application}.

In the sequel of this work (part II \cite{simonis2023hlbmPartII}), we construct and validate HLBMs to approximate the HNSE for porous media flow (see Figure~\ref{fig:limitCons}) within the framework of limit consistency introduced in \cite{simonis2022limit} and motivated already in \cite{krause2010fluid}. 
Therein, based on determining the truncation errors of governing families of equations at each level of discretization of the HBGKBE, the limit consistency of order two and one of the HLBE for the pressure and velocity, respectively, of the homogenized NSE is proven. 
In addition, HLBM simulations in various parameter regimes are conducted, numerically validating the present theoretical predictions.

\textbf{\emph{Acknowledgements:}}
S.\ Simonis would like to thank Fabian Klemens for valuable discussions.

\textbf{\emph{Funding:}}
This work was supported by the Deutsche Forschungsgemeinschaft (DFG, German Research Foundation, DOI: \href{http://dx.doi.org/10.13039/501100001659}{10.13039/501100001659}), project number \href{https://gepris.dfg.de/gepris/projekt/382064892?context=projekt&task=showDetail&id=382064892&}{382064892/SPP2045} as well as project number \href{https://gepris.dfg.de/gepris/projekt/468824876}{468824876}. 

\textbf{\emph{Author contribution statement:}}
\textbf{S.\ Simonis}: Conceptualization, Methodology, Validation, Formal analysis, Investigation, Writing - Original Draft, Writing - Review \& Editing, Visualization, Supervision, Project administration;
\textbf{N.\ Hafen}: Methodology, Validation, Investigation, Writing - Review \& Editing, Funding acquisition; 
\textbf{J.\ Je{\ss}berger}: Writing - Review \& Editing, Methodology, Formal analysis; 
\textbf{D.\ Dapelo}: Writing - Review \& Editing; 
\textbf{G.\ Th\"{a}ter}: Writing - Review \& Editing, Supervision; 
\textbf{M.\ J.\ Krause}: Resources, Writing - Review \& Editing, Supervision, Funding acquisition. 
All authors read and approved the final version of the manuscript.

%\bibliographystyle{abbrvnat}
%\bibliography{references}
%

\end{document}